\documentclass{amsart}
\usepackage[utf8]{inputenc}
\usepackage{xcolor,fullpage,hyperref}
\usepackage{tikz,amsmath,amssymb,color}
\usetikzlibrary{calc}
\usepackage{changepage}
\usepackage{tabularray}
\usepackage[longtable]{multirow}
\usepackage{longtable}
\usepackage{pgfplots}
\pgfplotsset{compat=1.15}
\usepackage{mathrsfs}
\usetikzlibrary{arrows}
\usepackage{ wasysym }

\usepackage{pgf}
\usepackage[backrefs]{amsrefs}
\usepackage{xcolor}
\usepackage{cleveref}
\usepackage{amssymb}

\newcommand{\defterm}{\textbf}
\newcommand{\Flatten}[0]{\mathrm{Flatten}}
\usepackage{amsmath,amsthm}
\theoremstyle{definition} 
\newtheorem{theorem}{Theorem}[section]
\newtheorem*{theorem*}{Theorem}
\newtheorem{corollary}{Corollary}[section]
\newtheorem{proposition}{Proposition}[section]

\newtheorem{example}{Example}[section]

\newtheorem{lemma}{Lemma}[section]

\theoremstyle{definition}
\newtheorem{definition}{Definition}[section]

\newcommand{\seqnum}[1]{\href{https://oeis.org/#1}{#1}}

\theoremstyle{remark}
\newtheorem{remark}{Remark}[section]

\newtheorem{question}{Question}[section]

\newcommand{\MFR}{\mathrm{MFR}}
\newcommand{\RFR}{\mathrm{RFR}}

\newcommand{\PF}[1]{\mathrm{PF}_{#1}}

\newcommand{\Fub}{\mathrm{Fub}}

\newcommand{\NN}{\mathbb{N}}

\newcommand{\FR}[1]{\mathrm{FR}_{#1}}
\newcommand{\UFR}[1]{\mathrm{UFR}_{#1}}

\newcommand{\out}{\mathcal{O}}

\newcommand{\len}{\mathrm{len}}
\newcommand{\Comp}{\mathrm{Comp}}

\newcommand{\oeis}[1]{\href{http://oeis.org/#1}}

\newcommand{\Sym}{\mathfrak{S}}
\usepackage{dsfont}

\usepackage{tikz,multicol,enumitem}

\title{Counting $\ell$-Interval Fubini Rankings\\Through their Parking Outcomes}

\author[Ager-Hart]{Bjorn Andreas Ager-Hart}
\author[Beerbower]{Melissa Beerbower}
\author[Harris]{Pamela E. Harris}
\author[Jakovleski]{Joakim Jakovleski}
\author[McClinton]{Matt McClinton}
\address[Ager-Hart, Beerbower, Harris, Jakovleski, McClinton]{Department of Mathematical Sciences, University of Wisconsin-Milwaukee, Milwaukee, WI 53211}
\email{
\textcolor{blue}{\href{mailto:agerhar2@uwm.edu}{agerhar2@uwm.edu}},
\textcolor{blue}{\href{mailto:beerbow2@uwm.edu}{beerbow2@uwm.edu}},
\textcolor{blue}{\href{mailto:peharris@uwm.edu}{peharris@uwm.edu}},
\textcolor{blue}{\href{mailto:joakim@uwm.edu}{joakim@uwm.edu}},
\textcolor{blue}{\href{mailto:mcclin33@uwm.edu}{mcclin33@uwm.edu}}
}

\begin{document}

\begin{abstract}
    Fubini rankings with $n$ competitors are $n$-tuples with entries in $[n]=\{1,2,3,\ldots, n\}$ that encode the conclusion of a race that allows ties. 
    Since Fubini rankings are parking functions, we can  study their parking outcomes, which are permutations encoding the final parking order of the cars using the Fubini ranking as a preference list. 
    We establish that the number of Fubini rankings with $n$ competitors having a fixed parking outcome $\pi$ is given by $2^{n-k}$, where $k$ denotes the number of runs in $\pi$. 
    We then use this formula to give a new proof for the number of Fubini rankings, which is given by the Fubini numbers. 
    We also consider the set of $\ell$-interval Fubini rankings with $n$ competitors, which are Fubini rankings where at most $\ell+1$ competitors tie at any rank. 
    We show that the number of $\ell$-interval Fubini rankings with $n$ competitors having a fixed parking outcome $\pi$ is given by a product of a power of two and a product of $\ell$-Pingala numbers, where these factors depend only on the lengths of the runs that make up the parking outcome $\pi$.
    The $1$-interval Fubini rankings are known as unit Fubini rankings, and we show that the number of unit Fubini rankings having a fixed parking outcome $\pi$ is given by a product of Fibonacci numbers indexed by the lengths of the runs in $\pi$. 
    We use these results to give a formula for the number of $\ell$-interval Fubini rankings with $n$ competitors for all $\ell\in[n]$. 
    We conclude with some directions for further study.
\end{abstract}
\maketitle

\section{Introduction}
Throughout, we let $\NN=\{1,2,3,\ldots\}$ denote the set of positive integers, and for any $n\in \NN$ we let $[n]=\{1,2,3,\ldots,n\}$. The set of permutations of $[n]$ is denoted $\mathfrak{S}_n$ and we write permutations in one-line notation $\pi=\pi_1\pi_2\cdots\pi_n$. Given $\pi\in\mathfrak{S}_n$, the \defterm{runs} of $\pi$ are the maximal contiguous increasing subwords of $\pi$. For any run $\tau$ of $\pi$, we let $\len(\tau)$ denote the length of $\tau$. 
As an example, if $\pi=371246598\in\mathfrak{S}_9$, the runs of $\pi$ are $37,1246,59,$ and $8$, and have lengths $2,4,2,$ and $1$, respectively. 
Note that each run begins whenever $\pi\in\Sym_n$ has a \defterm{descent}, which is an index $i\in[n-1]$ at which $\pi_i>\pi_{i+1}$.

The number of permutations in $\mathfrak{S}_n$ with exactly $k$  descents (equivalently $k+1$ runs) is given by the \defterm{Eulerian numbers} $A(n,k)=|\{\pi\in\mathfrak{S}_n : \pi \text{ has } k \text{ descents}\}|$, which were first studied in 1755 by Euler \cite{Euler}.
Comtet \cite[p.\ 243]{Comtet} credits Worpitzky \cite{Worpitzky} with the following explicit formula for the Eulerian numbers: 
\begin{align}
A(n,k)=\sum _{i=0}^{k}(-1)^{i}{\binom {n+1}{i}}(k+1-i)^{n}.
\end{align}

We now introduce Fubini rankings.
We consider a race of $n$ competitors in which ties are allowed; for a history of this problem we recommend \cite{Races_with_ties}. 
At the conclusion of the race,  we track the competitors and their rank as follows: for each competitor $i\in[n]$, we let $r_i$ denote the rank that competitor $i$ placed in at the conclusion of the race. 
Since ties are allowed, we can have 
multiple competitors with the same rank. 
However, if $k$ competitors are tied and place in rank $j$, then the $k-1$ subsequent ranks $j + 1, j + 2, \ldots, j + k - 1$ are disallowed, and the next possible rank (if it exists) is $j+k$.
Moreover, there is always at least one competitor placing in rank 1.
Whenever the tuple $\alpha=(a_1,a_2,\ldots,a_n)\in[n]^n$ satisfies all of these conditions, we say that $\alpha$
is a \defterm{Fubini ranking} with $n$ competitors, and we denote the set of Fubini rankings with $n$ competitors by $\FR{n}$.
For example, $(1,3,1)$ is a Fubini ranking with 3 competitors, in which competitors 1 and 3 tie for rank 1 and competitor 2 places in the next available rank, which is rank 3. 
However, $(2,1,1)$ is not a Fubini ranking, as the tie for rank 1 would disallow rank 2. 
In \Cref{fig:Fubini ex}, we illustrate the 13 Fubini rankings with three competitors.
\begin{figure}
\input{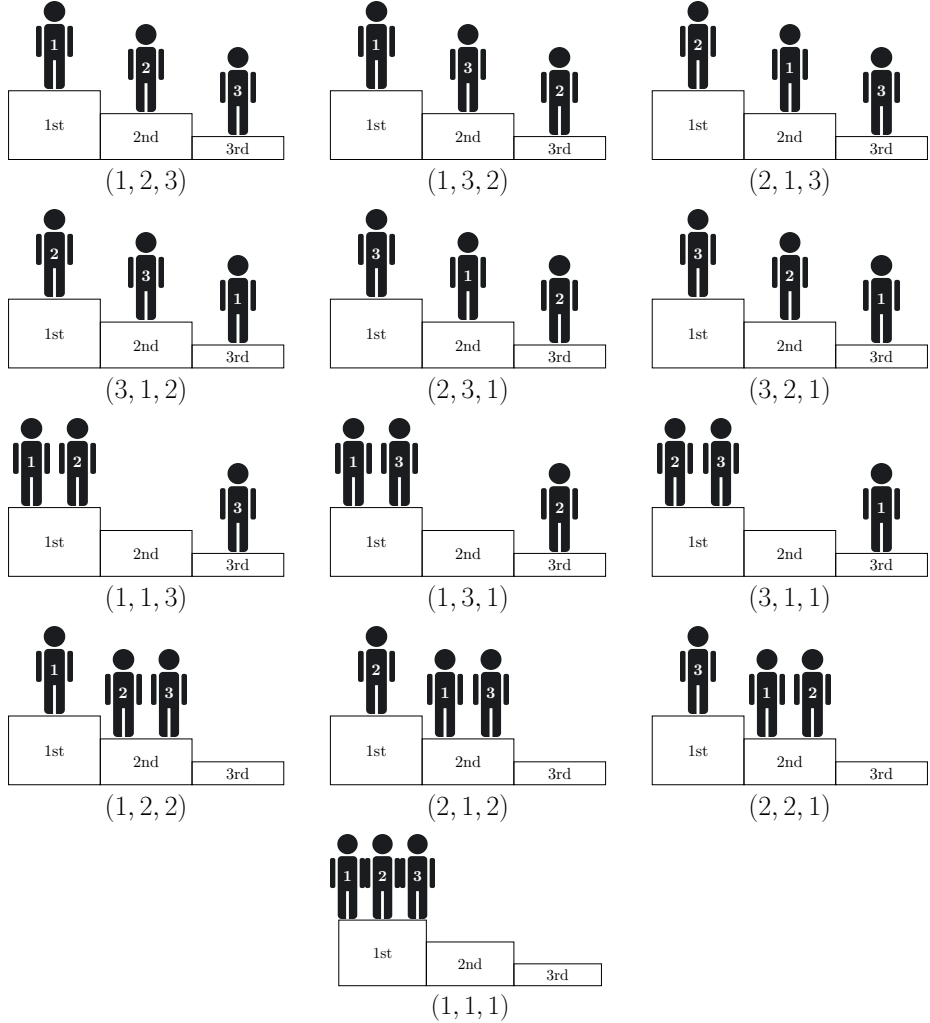}
    \caption{The 13 possible results of three competitors competing in a race allowing ties.}
    \label{fig:Fubini ex}
\end{figure}

Fubini rankings are enumerated by the \defterm{Fubini numbers}, also known as the \defterm{ordered Bell numbers}, and we denote the sequence by $\{\Fub_n\}_{n\geq 1}$.
The Fubini numbers are sequence \seqnum{A000670} in the OEIS, and satisfy 
\begin{align}\label{eq:formula for Fubs}
\displaystyle \Fub_n=\sum _{k=0}^{n-1}2^{k}\cdot A(n,k).
\end{align}

Fubini numbers arise in a variety of contexts in combinatorics. 
They count the number of possible passwords for a bike lock with buttons labeled $1$ through $n$, where entering a password requires one to push all buttons exactly once and where multiple buttons can be pressed simultaneously \cite{Permutations_and_combination_locks}. 
The Fubini numbers also arise in Conway's Napkin Problem \cite{napkin}, where guests are seated at a circular table one at a time, and have a preference for a napkin to their right or left. 
Moreover, there has been much work in studying Fubini rankings in connection to parking functions, see for example  \cite{weak,rFubini}, and we continue this line of study.
To make our approach precise, we begin by providing some background information on parking functions. 

A tuple $\alpha=(a_1,a_2,\ldots,a_n)\in[n]^n$ is a \defterm{parking function of length $\boldsymbol{n}$} if and only if the weakly increasing rearrangement of $\alpha$, denoted $\alpha^{\uparrow}=(a_1',a_2',\ldots,a_n')$ 
satisfies $a_i'\leq i$ for all $i\in[n]$. We denote the set of all parking functions of length $n$ by $\PF{n}$. 
Parking functions were named as such by Konheim and Weiss \cite{Konheim1966AnOD} as they can also be described via a parking procedure where $n$ cars are in a queue to park on a one-way street with $n$ parking spots. 
Cars enter the street sequentially, first attempting to park in their preferred spot. 
If that spot is occupied, then they park in the first available spot they encounter (if any). 
For each $i\in[n]$, if car $i$ has preference $a_i$  (i.e., 
$\alpha=(a_1,a_2,\ldots,a_n)$) and all cars are able to park, then we say $\alpha$ is a parking function of length $n$.
For example, the tuple $\alpha=(4,1,1,1,5)$ is a parking function of length $5$ and the final order in which the cars park on the street is illustrated in \Cref{fig:parking 41115}. 
The final parking order of the cars on the street is referred to as the \defterm{parking outcome} of $\alpha$, denoted $\out(\alpha)$. 
The parking outcome of a parking function is always a permutation of $[n]$, which we write in one-line notation.
In other words, $\out(\alpha)=\pi=\pi_1\pi_2\cdots\pi_n\in\Sym_n$, denotes that, for each $i\in[n]$, car $\pi_i$ parked in spot $i$. 
In our example, $\out(\alpha)=\out((4,1,1,1,5))=23415$.

\begin{figure}
         \centering 
         \resizebox{\textwidth}{!}{
         \begin{tikzpicture}
         \node at(0,0){\includegraphics[width=1in]{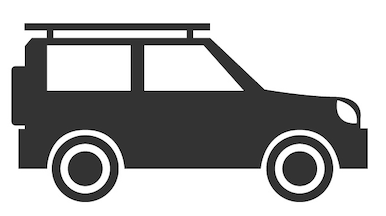}}; 
         \node at(3,0)
         {\includegraphics[width=1in]{car.png}}; \node at(6,0)
         {\includegraphics[width=1in]{car.png}}; \node at(9,0)
         {\includegraphics[width=1in]{car.png}}; \node at(12,0)
          {\includegraphics[width=1in]{car.png}};
         \node at(0,-.15){\textcolor{white}{\textbf{2}}}; 
         \node at(3,-.15){\textcolor{white}{\textbf{3}}}; 
         \node at(6,-.15){\textcolor{white}{\textbf{4}}}; 
         \node at(9,-.15){\textcolor{white}{\textbf{1}}}; 
         \node at(12,-.15){\textcolor{white}{\textbf{5}}}; 
         \node at(.9,1.)
         {\includegraphics[width=.5in]{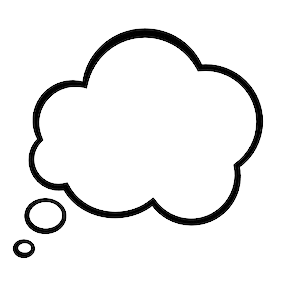}}; 
         \node at(3.9,1.)
         {\includegraphics[width=.5in]{callout.png}}; 
         \node at(6.9,1.)
         {\includegraphics[width=.5in]{callout.png}}; 
         \node at(9.9,1.)
         {\includegraphics[width=.5in]{callout.png}}; 
         \node at(12.9,1.)
         {\includegraphics[width=.5in]{callout.png}}; 
         \node at(.9,1.1){\textbf{1}}; 
         \node at(3.9,1.1){\textbf{1}}; 
         \node at(6.9,1.1){\textbf{1}}; 
         \node at(9.9,1.1){\textbf{4}}; 
         \node at(12.9,1.1){\textbf{5}}; 
         \draw[ultra thick](-1.25,-.6)--(1.25,-.6); 
         \draw[ultra thick](1.75,-.6)--(4.25,-.6); 
         \draw[ultra thick](4.75,-.6)--(7.25,-.6); 
         \draw[ultra thick](7.75,-.6)--(10.25,-.6); 
         \draw[ultra thick](10.75,-.6)--(13.25,-.6); 
         \node at(0,-1){\textbf{1}}; 
         \node at(3,-1){\textbf{2}}; 
         \node at(6,-1){\textbf{3}}; 
         \node at(9,-1){\textbf{4}}; 
         \node at(12,-1){\textbf{5}}; 
         \end{tikzpicture} 
         }
    \caption{Parking cars using the preferences in $\alpha=(4,1,1,1,5)$. Car image designed by Freepik and callout designed by macrovector / Freepik.}
    \label{fig:parking 41115}
\end{figure}

All Fubini rankings are parking functions, see \cite[Lemma 2.3]{rFubini}, and our goal is to enumerate Fubini rankings via their parking outcomes. 
To this end, 
for $\alpha\in\FR{n}$ we let $\out_{\FR{n}}(\alpha)$ denote the outcome of $\alpha$ using the parking procedure. 

Then, for a fixed permutation $\pi \in \Sym_n$, we let 
\[\out_{\FR{n}}^{-1}(\pi)=\{\alpha\in\FR{n}:\out_{\FR{n}}(\alpha)=\pi\},\]
which is the set of Fubini rankings with $n$ competitors that, when treated as parking functions, have final parking outcome $\pi$.
In \Cref{sec:FRs}, we establish  properties of Fubini rankings including equivalent characterizations (see \Cref{prop:equivalent characterizations of Fubini}), and give a way to describe Fubini rankings as ordered set partitions (see \Cref{lem:bijection between Fubini rankings and ordered set partitions}). Using these properties, we establish that the number of Fubini rankings with parking outcome $\pi\in\Sym_n$ is $2^{n-k}$, where $k$ is the number of runs of $\pi$ (see \Cref{cor: fubini count is 2^n-k}). We also establish that the number of Fubini rankings is given by the sum 
    \begin{equation}\label{eq:main results together}
\Fub_n=\sum_{\pi\in\mathfrak{S}_n}|\out_{\FR{n}}^{-1}(\pi)|=\sum_{k=1}^{n}2^{n-k}\cdot A(n,k-1) 
\end{equation}
    (see \Cref{corr: fubini numbers are odd}).
    The equivalence between \Cref{eq:formula for Fubs} and \Cref{eq:main results together} follows from the symmetry of the Eulerian numbers; $A(n,k)=A(n,n-1-k)$.
Since the number of Fubini rankings with a fixed parking outcome depends only on the number of runs of $\pi$, and since $A(n,k-1)$ counts the number of permutations with exactly $k$ runs (equivalently $k-1$ descents), our results establish a combinatorial proof of \Cref{eq:formula for Fubs}.

We also study $\ell$-interval Fubini rankings for $0\leq \ell\leq n-1$.
Defined by Aguilar-Fraga et al.\ \cite[p.\ 16]{aguilarfraga2024interval}
$\ell$-interval Fubini rankings with $n$ competitors are the subset of Fubini rankings with $n$ competitors in which at most $\ell+1$ (with $0\leq \ell\leq n-1$) competitors tie at any single rank. 
For example, $(1,1,1,4,5)$ is an $\ell$-interval Fubini ranking for $\ell \geq 2$, but not for $\ell < 2$. 
We let $\FR{n}(\ell)$ be the set of $\ell$-interval Fubini rankings with $n$ competitors. 
Note $\FR{n}(0)=\Sym_n$, where we write the permutations as tuples. Also, the set $\FR{n}(1)$ is referred to as the set of unit Fubini rankings, which have connections to counting Boolean intervals in the weak Bruhat order of the symmetric group \cite{weak}. 
We let $\FR{n}^\uparrow(\ell)$ denote the weakly increasing $\ell$-interval Fubini rankings with $n$ competitors.

To achieve our goal of counting $\ell$-interval Fubini rankings, we count $\ell$-interval Fubini rankings with $n$ competitors having a fixed parking outcome. 
To this end, 
we define the $\ell$-Pingala numbers, $\{P_\ell(n)\}_{n\geq1}$, via an $(\ell+1)$-term recurrence where the seed values are increasing powers of two (see \Cref{def:ell-Pingala}). 
If $\ell=1$, then $P_1(n)=F_n$, where $F_n$ is the $n$-th Fibonacci number with seed values $F_1=1$ and $F_2=1$. 
Aguilar-Fraga et al.\ \cite{aguilarfraga2024interval} showed that $|\FR{n}^\uparrow(\ell)|=P_\ell(n)$.

Using this result, we determine a product formula for the number of $\ell$-interval Fubini rankings with $n$ competitors whose parking outcome is $\pi\in\Sym_n$ (see \Cref{thm:big theorem}). 
The product formula involves a power of two and a product of  $\ell$-Pingala numbers, where both the power of two and the Pingala numbers appearing in this product depend only on the lengths of the runs that make up the parking outcome.
In the special case of $\ell=1$, the product described involves only Fibonacci numbers (see \Cref{thm:number of UFR with outcome pi}).

We now describe the organization of this article. In  \Cref{sec:FRs}, we provide background on Fubini rankings and some preliminary results needed to establish our count for Fubini rankings with a fixed parking outcome.
In \Cref{sec:UFRs}, we enumerate unit Fubini rankings with $n$ competitors having a fixed parking outcome and we use this result to give a new formula for the number of unit Fubini rankings.
In \Cref{sec:ell FRs}, we enumerate 
$\ell$-interval 
Fubini rankings with $n$ competitors having a fixed parking outcome and we use this result to give a new formula for the number of $\ell$-interval Fubini rankings with $n$ competitors.
We conclude with \Cref{sec:future} where we detail some open problems for further study.

\section{Enumerating Fubini Rankings by their Parking Outcome}\label{sec:FRs}
In this section, we provide some foundational results regarding Fubini rankings that are useful in enumerating them based on their parking outcome. 
In order to clarify our set-up, we make the following remark.
\begin{remark}
    When we study Fubini rankings we take on two perspectives simultaneously. 
    We think of Fubini rankings in terms of their definition as tuples encoding the ranks for a competition that allows ties, and in this context we use the words ``competitors'' and ``ranks.'' 
    We also think of Fubini rankings as parking functions, in which case ``competitors'' are equivalent to ``cars'' and ``ranks'' are equivalent to ``parking preferences.'' Moreover, the parking outcome of a Fubini ranking, is the final order of the cars using the parking process defining parking functions. 
    Henceforth, we use these two equivalent contexts and language interchangeably, based on whichever we find more appropriate to our purpose.
\end{remark}

We begin with the technical definition of Fubini rankings as presented in \cite[Definition 2.4]{weak}.

\begin{definition}\label{def:Fubini rankings}
    A \defterm{Fubini ranking} with $n$ competitors is a tuple $\alpha=(a_1,a_2,\ldots,a_n)\in [n]^n$ that records a valid ranking over $n$ competitors with ties allowed (i.e., multiple competitors can be tied and have the same rank).
    However, if $k$ competitors are tied at rank $i$, then the $k-1$ subsequent ranks $i + 1, i + 2, \ldots, i + k - 1$ are disallowed.
\end{definition}

Next, we introduce a means of verifying that a tuple is a Fubini ranking. 
\begin{lemma}[Characterization of Fubini rankings]\label{lemma: weakly increasing Fubini proof}
    Let $\alpha=(a_1,a_2,\ldots, a_n)\in[n]^n$ 
    and denote the weakly increasing rearrangement of $\alpha$ by $\alpha^\uparrow=(a'_1,a'_2,\ldots, a'_n)$, where $a_i'\leq a_{i+1}'$ for all $i\in[n-1]$. 
    Then $\alpha\in \FR{n}$ if and only if  for any $i\in [n]$, 
    either $a'_i=i$ or $a'_i=a'_{i-1}$. 
\end{lemma}
\begin{proof}
    $(\implies)$ Let $\alpha\in\FR{n}$.  
    For a ranking to be valid, there has to be at least one competitor ranking first. Thus, there is some number $x \in [n]$ such that $a_x = 1$. Since 1 is the smallest rank available, $a'_1=1$ in the rearrangement.
    By definition, for any Fubini ranking, if $k_1$ competitors place first, then the next available rank is $k_1+1$. If $k_1$ competitors place first, and $k_2$ competitors place in rank $k_1+1$, then the next available rank is $k_2+k_1+1$. 
    If we continue in this way, for any Fubini ranking $\alpha$, when $i\geq 2$ the $i$-th competitor in the weakly increasing rearrangement $\alpha^\uparrow$  has two options:
    \begin{itemize}
        \item they either continue a tie with some other competitor from the original Fubini ranking $\alpha$, which would mean $a'_{i} = a'_{i-1}$, or
        \item they are the first competitor in a new rank from the original Fubini ranking. Since $\alpha^\uparrow$ is sorted, we know this rank is not in the set $[i-1]$. Hence, the $i$th competitor will take the next available rank, which is $i$. Therefore, $a_i' = i$.
    \end{itemize} 
    So, we can see that for any $\alpha\in \FR{n}$ rearranged in weakly increasing order, the first entry satisfies $a'_1=1$, and the $i$th entry (for $i\geq 2$) either satisfies $a'_i=a'_{i-1}$ or $a'_i$ begins a new rank, meaning it satisfies $a_i'=i$. 
    
    \noindent $(\impliedby)$ Let $\alpha=(a_1,a_2,\ldots, a_n)\in[n]^n$ such that $\alpha^\uparrow$, the weakly increasing rearrangement of $\alpha$, satisfies either $a'_i=a'_{i-1}$ or $a'_i=i$ for each $i\in[n]$.
    We want to show that a tuple with this property is a Fubini ranking. 
    We recall that parking functions are permutation invariant. That is, given any permutation $\sigma\in\Sym_{n}$
and any parking function $\alpha=(a_1,a_2,\ldots,a_n)\in\PF{n}$, the action of $\sigma$ on $\alpha$ defined by 
\[\sigma (\alpha)=(a_{\sigma_1^{-1}},a_{\sigma_2^{-1}},\ldots,a_{\sigma_n^{-1}})\]
returns a parking function. For a formal proof of this fact see \cite[Theorem 3.2]{costsharing}.
    Since Fubini rankings are parking functions, they too are permutation invariant. Thus it is irrelevant which competitor places in which rank so long as the ranking is valid. 
    So to prove the claim it suffices to show that $\alpha^\uparrow$ is a Fubini ranking.
    To do so, by \Cref{def:Fubini rankings}, it suffices to show that whenever $k$ competitors tie at rank $i$, then the ranks $i+1,i+2,\ldots,i+k-1$ are disallowed, i.e., are not in $\alpha$.

    Let $\textrm{content}(\alpha)=\{i: i\in \alpha\}=\{r_1,r_2,\ldots,r_t\}\subseteq[n]$ and assume 
    $r_1<r_2<\cdots<r_t$.
    Let $k_i$ be the number of instances of $i$ in $\alpha$ for each $i\in \textrm{content}(\alpha)$.
    By the condition defining $\alpha$, we know that $a'_1 = 1$, so $k_1 \geq 1$ and $r_1=1$.
    Moreover, the first $k_1$ entries in $\alpha^\uparrow$ are 1. 
    By our condition on the tuple $\alpha^\uparrow$, we know that 
    $a'_{k_1+1}=k_1+1$, as we only have $k_1$ ones appearing in indices $1,2,\ldots, k_1$ in $\alpha^\uparrow$.
This implies that $r_2=k_1+1$ and the ranks $2,3,\ldots,k_1$ do not appear in $\alpha^\uparrow$ nor in $\alpha$.

    Now there are $k_2$ instances of the rank $r_2=k_1+1$ appearing in $\alpha^\uparrow$ (and hence in $\alpha$).
    This means that entries $a'_{k_1+1}=a'_{k_1+2}=\cdots=a'_{k_1+k_2}=r_2=k_1+1$ in 
    $\alpha^{\uparrow}$.
    By our condition on the tuple $\alpha^\uparrow$, we know that 
    $a'_{k_1+k_2+1}=k_1+k_2+1$, as we only have $k_2$ instances of the value $r_2$ appearing in indices $k_1+1,k_1+2,\ldots, k_1+k_2+1$ in $\alpha^\uparrow$.
    This implies that $r_3=k_1+k_2+1$ and the ranks $k_1+2,k_1+3,\ldots,k_1+k_2$ do not appear in $\alpha^\uparrow$ nor in $\alpha$.

    For induction, assume that for $j\leq t-1$, $r_{j}=1+\sum_{m=1}^{j-1}k_m$, and 
    \[a'_{1+\sum_{m=1}^{j-1}k_m}=a'_{2+\sum_{m=1}^{j-1}k_m}=\cdots=a'_{\sum_{m=1}^{j}k_m}=r_j,\]
    and that the ranks $2+\sum_{m=1}^{j-1}k_m,3+\sum_{m=1}^{j-1}k_m,\ldots,\sum_{m=1}^{j}k_m$ do not appear in $\alpha^\uparrow$ nor in $\alpha$.

    Now consider rank $r_{j+1}$.
    In $\alpha^\uparrow$, we have that $a'_{r_{j+1}}=r_{j+1}$, which is the first time we see the rank $r_{j+1}$ in $\alpha^\uparrow$.
    Since the ranks are listed in increasing order ($1=r_1<r_2<\cdots<r_j<r_{j+1}$), the number of times $r_i$ appears in $\alpha^\uparrow$ is $k_i$, and $\alpha^\uparrow$ is sorted, we have that $r_{j+1}=1+\sum_{m=1}^{j}k_m$.
    Hence, we know that the $k_{j+1}$ instances of the rank $r_{j+1}$ appear in indices $1+\sum_{m=1}^{j}k_m,2+\sum_{m=1}^{j}k_m,\ldots, \sum_{m=1}^{j+1}k_m$ in $\alpha^\uparrow$.
    This implies that the ranks $2+\sum_{m=1}^{j}k_m,3+\sum_{m=1}^{j}k_m,\ldots,\sum_{m=1}^{j+1}k_m$ do not appear in $\alpha^\uparrow$ nor in $\alpha$.

    This shows that whenever there are $k$ competitors tied at rank $i$, the correct set of ranks are disallowed (see \Cref{def:Fubini rankings}), so $\alpha^\uparrow $ is a Fubini ranking. 
    As we noted, since Fubini rankings are invariant under permutations, $\alpha$ is also a Fubini ranking, which completes the proof.
\end{proof}
To highlight the use of \Cref{lemma: weakly increasing Fubini proof} we provide the following example.

\begin{example}
    Consider the tuple $\alpha=(5,1,2,5,3,1)$. Hence $\alpha^\uparrow=(1,1,2,3,5,5)$. Notice 
     $a'_{3}\neq a_2'$ and $a'_{3}\neq3$. So, by \Cref{lemma: weakly increasing Fubini proof}, $(5,1,2,5,3,1)$ is not a Fubini ranking. This can also be observed since the tie at rank 1 would disallow rank 2.
\end{example}

Next, we recall that an \defterm{ordered set partition} of $[n]$ is an ordered collection of nonempty subsets (also called blocks) $(B_1,B_2,\ldots,B_k)$ such that $\cup_{i=1}^kB_i=[n]$ and we write the elements in each block in increasing order.
Fubini rankings are in bijection with ordered set partitions and we recall that bijection below.
\begin{lemma}[Theorem 3.13 \cite{barreto2025restrictedfubinirankingsrestricted}]\label{lem:bijection between Fubini rankings and ordered set partitions}
Let $\mathcal{OP}_n$ be the collection of ordered set partitions of $[n]$. 
Define $\psi:\FR{n}\to \mathcal{OP}_n$ as follows: Given $\alpha \in \FR{n}$,
\[\psi(\alpha)=(B_1,B_2,\ldots,B_t),\]
where $t$ denotes the number of distinct ranks appearing in $\alpha$, which we denote by $r_1<r_2<\cdots<r_t$, and where $B_j=\{i:a_i=r_j\}$ and the entries in $B_j$ are written in increasing order. Then $\psi$ is a bijection.
\end{lemma}

\begin{example}\label{Surprise!}
The bijection $\psi$ maps the Fubini ranking $\alpha=(5,1,5,2,2,2)$ to  the ordered set partition $\psi(\alpha)=(\{2\},\{4,5,6\},\{1,3\})$. 
Notice that $\out_{\FR{n}}((5,1,5,2,2,2))=245613$, which is the same as the ordered set partition removing all braces, commas, and parentheses. 
\end{example} 

The concluding remark in \Cref{Surprise!} that the outcome of a Fubini ranking is the ordered set partition with all braces, commas, and parenthesis removed is not a coincidence, and we prove this in \Cref{thm:ordered set partition is the outcome}. However, to establish that proof  we need the following preliminary results.

\begin{lemma}[Ranks Lemma]\label{lemma: ranks determined by number of tied competitors}
    Let $\alpha\in\FR{n}$ have ranks $1=r_1<r_2< \cdots< r_t$.
    Then $r_i = 1 + \sum_{j=1}^{i-1} |B_j|$.
\end{lemma}
\begin{proof}
  By \Cref{lem:bijection between Fubini rankings and ordered set partitions}, $\psi(\alpha)=(B_1,B_2,\ldots,B_t)$ where, for each $i\in[t]$, $|B_i|$ is the number of competitors tied at rank $r_i$. 
  By definition of a Fubini ranking, having $|B_1|$ competitors tied at rank 1 disallows ranks $2,3,\ldots,|B_1|$, so the next available rank is $r_2=1+|B_1|$. 
  
  For induction, assume that for all $k<t$, $r_k=1+\sum_{j=1}^{k-1}|B_j|$ and consider $r_{k+1}$. 
  There are $\sum_{j=1}^{k}|B_j|$, competitors that have placed in ranks $r_1,r_2,\ldots,r_{k}$. 
  Since the ranks are ordered from smallest to largest, we know that the competitors placing in ranks $r_1,r_2,\ldots,r_k$, which are those in the set $\cup_{j=1}^{k}B_j$, have 
  disallowed ranks $1,2,\ldots,\sum_{j=1}^k|B_j|$.
  Thus, the next available rank would be $r_{k+1}=1+\sum_{j=1}^k|B_j|$, as claimed. 
\end{proof}

We collect our characterizations of Fubini rankings in the following result.

\begin{proposition}\label{prop:equivalent characterizations of Fubini}
    Let $\alpha = (a_1, a_2,\ldots, a_n) \in [n]^n$, and let $\alpha^{\uparrow} = (a'_1,a_2', \ldots, a'_n)$ be its weakly increasing rearrangement. The following are equivalent:
    \begin{enumerate}
        \item[(a)] The tuple $\alpha$ is a Fubini ranking.
        \item[(b)] Let $\{r_1, r_2, \ldots, r_t\}$ be the unique entries of $\alpha$, sorted in increasing order ($r_{i-1} < r_i$ for any $i$), and for each $r_i$ let $k_i$ be the number of occurrences of $r_i$ in $\alpha$. Then, for every $i \in [t]$, $r_i = 1 + \sum_{j=1}^{i-1} k_j$. (Note $\sum_{j=1}^0 k_j$ is the empty sum, which equals 0).
        \item[(c)] For every $i \in [n]$, either $a'_i = i$ or $a'_i = a'_{i-1}$. 
    \end{enumerate}
\end{proposition}
\begin{proof}
We proceed by establishing the following implications:
    \begin{itemize}[leftmargin=.15in]
        \item (a)$\implies$(b): This is \Cref{lemma: ranks determined by number of tied competitors}. 
    
        \item (b)$\implies$(c): We want to show that for each rank $r_i$, the entries in $\alpha^{\uparrow}$ equal to rank $r_i$ respect one of the conditions of part (c). 

        We proceed by induction on $i$. Consider $i=1$. By our assumption, we know that $r_1 = 1 + \sum_{j=1}^{0} k_j = 1+0 = 1$, meaning that $a'_1 = 1$. Since there are $k_1$ instances of $r_1 = 1$ and $\alpha^{\uparrow}$ is weakly increasing, we have that $a'_1 = a'_2 = \cdots = a'_{k_1} = 1$. Therefore, for all integers $m$ in the interval $[2, k_1]$ we have $a'_m = a'_{m-1}$.  

        Assume, for induction, that $a'_m = a'_{m-1}$ or $a'_m = m$ whenever $a'_m \in \{r_1, r_2, \ldots, r_{i-1}\}$. Then this holds for $\sum_{j=1}^{i-1} k_j = r_i-1$ entries of $\alpha^{\uparrow}$.  In fact, since the $r_i$'s and $\alpha^{\uparrow}$ are sorted in (weakly) increasing order, we know that this holds for the \emph{first} $r_i-1$ entries of $\alpha^{\uparrow}$. We want to show that this also holds when $a'_m = r_i$. 
        
        By our inductive hypothesis, we know that $r_i = 1 + \sum_{j=1}^{i-1} k_j$. Since there are exactly $r_i-1$ entries of $\alpha^{\uparrow}$ smaller than $r_i$, this means that $a'_{r_i} = r_i$. Since there are $k_i$ instances of $r_i$ in $\alpha^{\uparrow}$ and $\alpha^{\uparrow}$ is weakly increasing, we have that $a'_{r_i} = a'_{r_i+1} = \cdots = a'_{r_i+k_i-1}$. 
        But this means that we have $a'_m = a'_{m-1}$ for every $m \in [r_i+1, r_i + k_i]$, and our inductive case is done.
    
        \item (c)$\implies$(a): This is the backwards direction of \Cref{lemma: weakly increasing Fubini proof}.\qedhere
    
    \end{itemize}
\end{proof}

Bradt et al.\ state  in \cite[Observation 2.4]{rFubini} that in a Fubini ranking the set of cars with the same rank ultimately park in sequential order on the street. 
We restate this observation and provide a proof.

\begin{lemma}[Tied players park in sequential order]\label{lemma: competitor outcome for ties}
    Let $\alpha\in\FR{n}$ with $\psi(\alpha)=(B_1,B_2,\ldots,B_t)$.
    Then for each $i\in[t]$ the competitors in block $B_i$ appear in lexicographic order in the parking outcome $\pi = \out_{\FR{n}}(\alpha)$, in the spots numbered \[1+\sum_{j=1}^{i-1}|B_j|,~2+\sum_{j=1}^{i-1}|B_j|,~\ldots,~ \sum_{j=1}^{i}|B_j|. \]
\end{lemma}
\begin{proof}
    Recall that if $\alpha=(a_1,a_2,\ldots,a_n)\in\FR{n}$ has ranks $1=r_1<r_2<\cdots<r_t$, where by \Cref{prop:equivalent characterizations of Fubini} $r_i=1+\sum_{j=1}^{i-1}|B_j|$ for each $i\in[t]$, then the corresponding ordered set partition is
    \[\psi(\alpha)=(B_1,B_2,\ldots,B_t),\]
    where $B_m=\{i:a_i=r_m\}$ for each $m\in[t]$.
    
    First consider $B_1=\{b_1,b_2,\ldots,b_{|B_1|}\}$
    (where we assume $b_1<b_2<\cdots<b_{|B_1|}$)
    which is the set of all competitors in $\alpha$ that placed first in the competition. 
    Since we order the entries in $B_1$ in increasing order,  the smallest entry in $B_1$ - namely, $b_1$ - parks in spot $1$, the next entry ($b_2$) parks in spot $2$, and so on until the largest entry in $B_1$ ($b_{|B_1|}$) parks in spot $|B_1|$.
    Therefore, the competitors in $B_1$ appear in lexicographic order in the spots numbered
    \[1 + \sum_{j=1}^0 |B_j|,~2 + \sum_{j=1}^0 |B_j|,~\ldots,~\sum_{j=1}^1 |B_j|, \]
    where we set $\sum_{j=1}^0|B_j|=0$ (an empty sum).

    By the definition of a Fubini ranking, the next rank is equal to $r_2=1+|B_1|$. 
    Thus all of the entries in $B_2=\{c_1,c_2,\ldots,c_{|B_2|}\}$, where we assume $c_1<c_2<\cdots<c_{|B_2|}$,  have  preference $r_2$ and 
car $c_1$ parks in spot $r_2$, car $c_2$ parks in spot $r_2+1$, and so on until competitor $c_{|B_2|}$ parks in spot $|B_1|+|B_2|$.
Therefore, the competitors in $B_2$ appear in lexicographic order in the spots numbered 
    \[1 + \sum_{j=1}^1 |B_j|,~2 + \sum_{j=1}^1 |B_j|,~\ldots,~|B_2| + \sum_{j=1}^{1}|B_j|, 
    \]
    and note $|B_2| + \sum_{j=1}^{1}|B_j|= \sum_{j=1}^2 |B_j|$.

By induction, assume that for all $k\leq i<t$ the cars in block $B_k$ park in spots 
\[1 + \sum_{j=1}^{k-1} |B_j|,~2 + \sum_{j=1}^{k-1} |B_j|,~\ldots, \sum_{j=1}^{k}|B_j|.\]

Now consider block $B_{i+1}$. By definition of a Fubini ranking, the competitors in $B_{i+1}=\{d_1,d_2,\ldots,d_{|B_{i+1}|}\}$, where we assume $d_1<d_2<\cdots<d_{|B_{i+1}|}$,   tie at rank 
$r_{i+1}=1+\sum_{j=1}^{i}|B_j|$. 
By the induction hypothesis, we know the cars in blocks $B_1,B_2,\ldots,B_i$ park in spots 
\[1,2,3,\ldots,\sum_{j=1}^{i}|B_j|.\]
Thus, the first car in $B_{i+1}$, which is $d_1$, parks in spot $1+\sum_{j=1}^{i}|B_j|$, the second car, $d_2$, parks in  $2+\sum_{j=1}^{i}|B_j|$, and so on until the last car $d_{|B_{i+1}|}$ parks in spot $\sum_{j=1}^{i+1}|B_j|$.
Therefore, the competitors in $B_{i+1}$ appear in lexicographic order in the spots numbered 
    \[1 + \sum_{j=1}^{i} |B_j|,~2 + \sum_{j=1}^{i} |B_j|,~\ldots,~\sum_{j=1}^{i+1}|B_j|,
    \]
as claimed. 
\end{proof}
 
Recall that if $\pi=\pi_1\pi_2\cdots\pi_n\in\Sym_n$, and a parking function has parking outcome $\pi$, then $\pi_i$ is the car parked in spot $i$ for each $i\in[n]$.
Given that $\pi$ is a permutation, its inverse $\pi^{-1}=\pi_1^{-1}\pi_2^{-1}\cdots\pi_n^{-1}$ has the following property: for each $j\in[n]$, $\pi_j^{-1}$ is the spot that car $j$ parked in.
For example, If $\pi=4213$, then car $\pi_1=4$ parked in spot 1, car $\pi_2=2$ parked in spot 2, car $\pi_3=1$ parked in spot 3, and car $\pi_4=3$ parked in spot 4. 
Since $\pi^{-1}=3241$, 
$\pi_1^{-1}=3$ tells us that car 1 parked in spot 3,
$\pi_2^{-1}=2$ tells us that car 2 parked in spot 2,
$\pi_3^{-1}=4$ tells us that car 3 parked in spot 4, and 
$\pi_4^{-1}=1$ tells us that car 4 parked in spot 1. 
In this way, for any permutation $\pi \in \Sym_n$, $\out(\pi)=\pi^{-1}$, see \cite[Theorem 2.2]{MVP_PFs} for a formal proof of this fact.
Our next result shows that for any two competitors earning different ranks, the one earning a smaller rank parks to the left of the one earning a larger rank. 

\begin{lemma}\label{lem:lemma with lots of parts}
    Let $\alpha=(a_1,a_2,\ldots,a_n)\in\FR{n}$ have parking outcome $\out_{\FR{n}}(\alpha)=\pi$. 
    Then, for any two competitors $i $ and $j$, if $a_{i} < a_{j}$, then $\pi_{i}^{-1} < \pi_{j}^{-1}$.
\end{lemma}
\begin{proof}
        Let $r_1<r_2<\cdots<r_t$ denote the ranks of $\alpha$, and assume $a_{i} = r_{x}$ and  $a_{j} = r_{y}$, for some $r_{x},r_{y}\in\{r_1,r_2,\ldots,r_t\}$. 
        Since we are assuming $a_{i} < a_{j}$, we have that $r_{x} < r_{y}$. 
        Moreover, the ranks are sorted from smallest to largest, so we have $r_{y} \geq r_{x+1} > r_{x+1}-1$. 
        By \Cref{lemma: competitor outcome for ties}, we have that $j$ parked in a spot between $r_{y}$ and $r_{y+1}-1$, and $i$ parked in a spot between $r_{x}$ and $r_{x+1}-1$. 
        Since $\pi_{i}^{-1}$ denotes the spot occupied by $i$ and $\pi_{j}^{-1}$ denotes the spot occupied by $j$, we can use the previous inequalities to  get 
        \[r_{x} \leq \pi_{i}^{-1} \leq r_{x+1}-1 < r_{x+1} \leq r_{y} \leq \pi_{j}^{-1} \leq r_{y+1}-1,\]
        which completes the proof.
\end{proof}

\begin{remark}
\Cref{lem:lemma with lots of parts} is not true for parking functions in general. For example, consider the parking function $\alpha=(1,1,1,4,2)$, which is not a Fubini ranking. 
If $i=5$ and $j=4$, then $a_{i}=a_5=2<4=a_{4}=a_{j}$. However,  $\out(\pi)=12345$ and $\pi_{i}^{-1}=\pi_5^{-1}=5  \nless 4=\pi_4^{-1}=\pi_{j}^{-1}$.
\end{remark}

Recall that \Cref{lem:bijection between Fubini rankings and ordered set partitions} gives a bijection between Fubini rankings and ordered set partitions. Next we show that the parking outcome of a Fubini ranking is precisely the ordered set partition it corresponds to after removing braces, commas, and parenthesis. 
To make this precise, we recall the Mathematica function $\Flatten$ (see \cite{flatten}), which flattens out nested lists. For example, $\Flatten((\{3,4,7\},\{1\},\{2,5,6\}))=3471256$.
    
\begin{theorem}\label{thm:ordered set partition is the outcome}
    If $\alpha\in\FR{n}$, then 
    $\out_{\FR{n}}(\alpha)=\Flatten(\psi(\alpha))$.
\end{theorem}
\begin{proof}
    Let $\alpha\in\FR{n}$ and let $\psi(\alpha)=(B_1,B_2,\ldots,B_t)$ be its corresponding ordered set partition, where we order the entries in each block in increasing order. 
    We also recall that each block of competitors corresponds to a rank in $\alpha$, and the blocks are sorted by increasing rank. Namely, if competitor $b_i$ is in block $B_i$, competitor $b_j$ is in block $B_j$, and $i < j$, then $a_{b_i} < a_{b_j}$ i.e., competitor $b_i$ ranks \textit{before} competitor $b_j$. 
    For each $i\in[t]$, label the elements in each block as $B_i=\{b_{(i,1)},b_{(i,2)},\ldots,b_{(i,|B_i|)}\}$, where we assume that $b_{(i,j)}<b_{(i,j+1)}$ for all $j\in[|B_i|-1]$. 
    Create the word $w=w_1w_2\cdots w_n=\Flatten(\psi(\alpha))$.

    Let
    $\out_{\FR{n}}(\alpha)=\pi=\pi_1\pi_2\cdots\pi_n\in\Sym_n$. 
    For each $i\in [n]$, the value $\pi_i$ is the label of the car parked in spot~$i$.
    We want to show that 
    $\pi_i=w_i$, for all $i\in[n]$.

    We proceed by induction on the index $i\in[n]$. 
    To start, recall that by \Cref{lemma: competitor outcome for ties} 
    the cars in block 
    $B_1=\{b_{(1,1)},b_{(1,2)},\ldots,b_{(1,|B_1|)}\}$ park in spots $1,2,\ldots,|B_1|$.
    Thus, $\pi_i=b_{(1,i)}$ for each $i\in[|B_1|]$. 
    Moreover, by construction of $w$, we have $w_i=b_{(1,i)}$ for each $i\in[|B_1|]$. By transitivity, $\pi_i=w_i$ for all $i\in[|B_1|]$.
    This establishes the base case of the induction on the first $|B_1|$ cars.
    
    Assume for induction that $\pi_i=w_i$ for all $i< n$.
    Now consider $\pi_{i+1}$ and $w_{i+1}$.
    Given $w_i$, we know there exists $j\in[t]$ and $s\in[|B_j|]$ such that 
    $w_{i}=b_{(j,s)}$, i.e., $w_i$ is the $s$-th entry in block $B_j$. 
    Then there are two possibilities for what $w_{i+1}$ could be:
    \begin{itemize}
        \item $w_{i+1}=b_{(j,s+1)}$ if $w_i$ was \textit{not} the last entry of a block, and 
        \item $w_{i+1}=b_{(j+1,1)}$ if $w_i$ was the last entry of a block.
    \end{itemize}
    In the case when $w_{i+1}=b_{(j,s+1)}$, we know by \Cref{lemma: competitor outcome for ties} that car $b_{(j,s)}$ parks in spot 
    $y=s+\sum_{x=1}^{j-1}|B_x|$,
    and so car $b_{(j,s+1)}$ parks in spot 
    $y+1=s+1+\sum_{x=1}^{j-1}|B_x|$.
    Hence $\pi_y=b_{(j,s)}$ and $\pi_{y+1}=b_{(j,s+1)}$.
    
    Now recall that, by assumption, we have $w_i=b_{(j,s)}$, which means that $i=s+\sum_{x=1}^{j-1}|B_x|$. Thus, $i=y$. 
    So by transitivity,  this establishes that 
    $\pi_{i+1}=w_{i+1}$, as desired.
    
    In the case when $w_{i+1}=b_{(j+1,1)}$, we know that $w_{i+1}$ is the first entry in block $B_{j+1}$ and, by \Cref{lemma: competitor outcome for ties}, car $b_{(j+1,1)}$ parks in spot 
    $z=1+\sum_{x=1}^j|B_x|$.
    Thus $\pi_{z}=b_{(j+1,1)}$.
    Now recall that by assumption $w_{i+1}=b_{(j+1,1)}$, which means that $i+1=1+\sum_{x=1}^{j}|B_x|$. Thus, $i+1=z$.  
    So by transitivity, this yields 
    $\pi_{i+1}=w_{i+1}$, which completes the proof.
\end{proof}
\begin{remark}
    \Cref{thm:ordered set partition is the outcome} is not true for parking functions in general. As an example, consider $\alpha=(1,2,1)\in\PF{3}$ and note that $\alpha\notin\FR{3}$. 
    Observe that $\out(\alpha)=123$ and if we extend the domain of the function $\psi$ to parking functions, the ordered set partition corresponding to $\alpha$ is $\psi(\alpha)=(\{1,3\},\{2\})$. Therefore $\Flatten(\psi(\alpha))=132\neq123= \out(\alpha)$.
\end{remark}

Next, for a fixed $\pi \in \Sym_n$, we describe a bijection between the subset of ordered set partitions whose ``flattening'' is $\pi$ and the set of Fubini rankings whose parking outcome is $\pi$.
To state this result, for a fixed $\pi\in\Sym_{n}$, we define 
\[\Flatten^{-1}(\pi)=\{B\in\mathcal{OP}_n:\Flatten(B)=\pi \}.\]
\begin{theorem}\label{thm: OSP preimage = flatten preimage}
    If $\pi\in\Sym_n$, then 
    $\psi(\out_{\FR{n}}^{-1}(\pi))=\Flatten^{-1}(\pi)$.
\end{theorem}
\begin{proof}
    Fix $\pi\in\Sym_n$. We establish both set containments.

    \noindent $(\subseteq)$ Let $\alpha\in\FR{n}$ be arbitrary  with $\out_{\FR{n}}(\alpha)=\pi$. 
    Then $\psi(\alpha)\in\mathcal{OP}_n$ and
    by \Cref{thm:ordered set partition is the outcome}, $\Flatten(\psi(\alpha))=\out_{\FR{n}}(\alpha)=\pi$.
    Thus $\psi(\alpha)\in\Flatten^{-1}(\pi)$, as desired.

    \noindent $(\supseteq)$ Let $\mathcal{B}\in\mathcal{OP}_n$ be arbitrary with $\Flatten(\mathcal{B})=\pi$. Since $\psi$ is a bijection, there exists a unique $\alpha\in\FR{n}$ with $\psi(\alpha)=\mathcal{B}$. By \Cref{thm:ordered set partition is the outcome}, $\Flatten(\mathcal{B})=\Flatten(\psi(\alpha))=\out_{\FR{n}}(\alpha)=\pi$. Thus $\mathcal{B}\in\psi(\out_{\FR{n}}^{-1}(\pi))$, which completes the proof.
\end{proof}
Next, we let $\FR{n}^\uparrow$ denote the set of weakly increasing Fubini rankings and show that this set is equivalent to the set of Fubini rankings whose parking outcome is the identity permutation $123\cdots n$.
\begin{lemma}\label{lem:weakly increasing Fubinis park in identity order}
    For all $n\in\mathbb{N}$, $\out^{-1}_{\FR{n}}(123\cdots n)=\FR{n}^\uparrow$ and $|\out^{-1}_{\FR{n}}(123\cdots n)|=|\FR{n}^\uparrow|=2^{n-1}$.
\end{lemma}
\begin{proof}
    Fix $n\in\NN$. We establish both set containments.
    
    \noindent $(\subseteq)$ Consider $\alpha=(a_1,a_2,\ldots,a_n)\in\out_{\FR{n}}^{-1}(123\cdots n)$. 
    For contradiction, 
    assume there exists an index $i\in[n-1]$ such that $a_i>a_{i+1}$ (so that $\alpha$ is not weakly increasing). 
    Then, in the corresponding ordered set partition $\psi(\alpha)$, there exist blocks $B_x$ and $B_y$, with $x<y$, such that $i+1\in B_x$ and $i\in B_y$. Then, by \Cref{thm: OSP preimage = flatten preimage}, the parking outcome of $\alpha$ will have $i+1$ appearing to the left of $i$, which contradicts $\out_{\FR{n}}(\alpha)=123\cdots n$. Thus $\alpha\in\FR{n}^\uparrow$.

    \noindent $(\supseteq)$ Consider $\alpha=(a_1,a_2,\ldots,a_n)\in \FR{n}^\uparrow$. We know that $a_1\leq a_2\leq \cdots\leq a_n$.
    Suppose that the ranks in $\alpha$ are $1=r_1<r_2<\cdots<r_t$.
    Then $\psi(\alpha)=(B_1,B_2,\ldots,B_t)$, 
    where $B_i = \{j: a_j = r_i\}$ for each $i\in[t]$.
   
    Since $\alpha$ is sorted in weakly increasing order, 
    any competitors tied at a rank would appear in consecutive entries in $\alpha$. 
    Hence, for each $i\in[t]$, 
    $B_i$ consists of a nonempty interval of consecutive integers in $[n]$ for any $i\in[t]$.
   
    Recall that by construction, we order the elements in a block $B_i$ in increasing order and, by definition of $\psi$, we order the blocks $B_i$ by their corresponding ranks. 
    Hence, the ordered set partition $(B_1,B_2,\ldots,B_t)$ lists all of the elements in $[n]$ in increasing order. 
    Thus 
    $\Flatten(\psi(\alpha))=123\cdots n$.
    By \Cref{thm:ordered set partition is the outcome}, $\Flatten(\psi(\alpha))=\out_{\FR{n}}(\alpha)=123\cdots n$. 
    This implies that $\alpha\in\out^{-1}_{\FR{n}}(123\cdots n)$.

For the cardinality result, let $\alpha=(a_1,a_2,\ldots,a_n)\in\FR{n}^{\uparrow}$. Recall that every Fubini ranking must contain a 1. Hence, by $\alpha$ being a weakly increasing Fubini ranking, $a_1=1$. Then, by \Cref{prop:equivalent characterizations of Fubini} part (c), for each $i\in[2,n]$ competitor $i$ with rank $r_i$ has two options: 
\begin{itemize}
    \item they tie with competitor $i-1$ (at a rank strictly less than $i$) so $a_i=a_{i-1}$, or
    \item they have rank equal to $i$, so $a_i=i$.
\end{itemize}

    This gives a total of $2^{n-1}$ weakly increasing Fubini rankings with $n$ competitors, as claimed.
\end{proof}

The next result shows that whenever two cars are in the same run in the parking outcome and have the same parking preference, every car parked between them has the same preference.

\begin{lemma}\label{lemma: equal ranks in runs}
    Let $\alpha=(a_1,a_2,\ldots,a_n)\in\FR{n}$ with parking outcome $\out_{\FR{n}}(\alpha)=\pi=\pi_1\pi_2\cdots\pi_n$ and  let $\pi=\tau_1\tau_2\cdots\tau_k$ where $\tau_1,\tau_2,\ldots,\tau_k$ are the runs of $\pi$. 
    If  $\pi_x,\pi_z\in\tau_i$, with $\pi_x < \pi_z$ for some $i\in[k]$, and $a_{\pi_x}=a_{\pi_z}$, then $a_{\pi_x}=a_{\pi_y}=a_{\pi_z}$ for all $\pi_y\in \tau_i$ with $\pi_x<\pi_y<\pi_z$. 
\end{lemma}
\begin{proof}
    Recall that if $\alpha=(a_1,a_2,\ldots,a_n)\in\FR{n}$ with ranks $1=r_1<r_2<\cdots<r_t$, where by \Cref{lemma: competitor outcome for ties}, $r_i=1+\sum_{j=1}^{i-1}|B_j|$ for each $i\in[t]$, then the corresponding ordered set partition is
    \[\psi(\alpha)=(B_1,B_2,\cdots,B_t),\]
    where $B_m=\{i:a_i=r_m\}$ for each $m\in[k]$.
    Assume that $\pi_x,\pi_z\in\tau_i$ for some $i\in[t]$ and that $a_{\pi_x}=a_{\pi_{z}}$.
    Furthermore, suppose that they tie at rank $r_m$, i.e., $a_{\pi_x}=a_{\pi_{z}}=r_m$.
    Then, by definition of the bijection $\psi$, we have that $\pi_x,\pi_z\in B_m$.
    
    By \Cref{lemma: competitor outcome for ties}, 
    $\pi_x$ and $\pi_z$ appear in lexicographic order in $\tau_i$ in two of the spots among those labeled
    \[1+\sum_{j=1}^{m-1}|B_j|, ~2+\sum_{j=1}^{m-1}|B_j|,~\ldots~, ~\sum_{j=1}^{m}|B_j|.\]
    Now suppose that $\pi_y\in \tau_i$ and $\pi_x <\pi_y < \pi_z$, and suppose for contradiction that $\pi_y\notin B_m$. 
Then there exists a block $B_{s}$ which contains $\pi_y$, and by \Cref{lemma: competitor outcome for ties}, $\pi_y$ must park in a spot among those labeled 
\[1+\sum_{j=1}^{s-1}|B_j|, ~2+\sum_{j=1}^{s-1}|B_j|,~\ldots~, ~\sum_{j=1}^{s}|B_j|.\]
Since $|B_j|>0$ for all $j$, this contradicts the assumption that $\pi_y$ parked strictly between $\pi_x$ and $\pi_z$.
Therefore $\pi_y\in B_m$, so by definition of the map $\psi$, $a_{\pi_y}=r_m=a_{\pi_x}=a_{\pi_z}$, as desired.
\end{proof}
 
The next result establishes that when cars are part of distinct runs in the parking outcome, they did not tie in the competition.  
\begin{lemma}\label{lemma: ties do not break over runs}
    Let $\alpha=(a_1,a_2,\ldots,a_n)\in\FR{n}$ with $\out_{\FR{n}}(\alpha)=\pi$ for some $\pi\in\mathfrak{S}_n$. Let $\pi=\tau_1\tau_2\cdots\tau_k$ and suppose $p_1\in\tau_i$ and $p_2\in\tau_j$ with $i<j$. Then $a_{p_1}\neq a_{p_2}$. 
\end{lemma}
\begin{proof} 
    Recall that if $\alpha=(a_1,a_2,\ldots,a_n)\in\FR{n}$ with ranks $1=r_1<r_2<\cdots<r_t$, where by \Cref{lemma: competitor outcome for ties} $r_i=1+\sum_{j=1}^{i-1}|B_j|$ for each $i\in[t]$, then the corresponding ordered set partition is
    \[\psi(\alpha)=(B_1,B_2,\ldots,B_t),\]
    where $B_m=\{i:a_i=r_m\}$ for each $m\in[t]$.
    Suppose for the sake of contradiction that $a_{p_1}=a_{p_2}=r_x$ for some $x\in [t]$, which implies that both $p_1,p_2\in B_x$. By \Cref{lemma: competitor outcome for ties}, the entries in $B_x$, and the entries $p_1$ and $p_2$ in particular, appear in $\pi$ in lexicographic order in spots labeled 
    \[1+\sum_{j=1}^{x-1}|B_j|, ~2 + \sum_{j=1}^{x-1}|B_j|, \ldots, \sum_{j=1}^{x}|B_j|. \]
    This implies that both $p_1$ and $p_2$ appear in the same run, which is a contradiction since we assumed that $p_1\in\tau_i$ and $p_2\in\tau_j$ with $i<j$. 
    Therefore, $a_{p_1}\neq a_{p_2}$.
\end{proof} 

Next, let $\pi \in \Sym_n$ and let $\alpha \in \FR{n}$ be \textit{any} Fubini ranking that satisfies $\out(\alpha)=\pi$. We show that if $\pi$ has a descent at index $i$, $\pi_{i}$ and $\pi_{i+1}$ are the last and first entries, respectively, of adjacent blocks of $\psi(\alpha)$.

\begin{lemma}\label{lem:descent force blocks}
Fix $\pi=\pi_1\pi_2\cdots\pi_n\in\Sym_n$ and let $\mathcal{B}=(B_1,B_2,\ldots,B_t)\in\Flatten^{-1}(\pi)$, where the entries in each block $B_i$ are listed in increasing order.
If $\pi_x>\pi_{x+1}$ for some index $x\in[n-1]$, then there exists $i\in[t-1]$ such that $\pi_x=\max(B_i)$ and $\pi_{x+1}=\min(B_{{i+1}})$.
\end{lemma}
\begin{proof}
    Assume $\pi_x>\pi_{x+1}$ with  $x\in[n-1]$. Assume for contradiction that $\pi_x$ and $\pi_{x+1}$ are in the same block $B_i$ of $\mathcal{B}$. Then, since the elements in $B_i$ are sorted in increasing order, $\pi_x$ would be to the right of $\pi_{x+1}$ in $\pi$. This contradicts the assumption that the entry $\pi_x$ is followed immediately by $\pi_{x+1}$ in $\pi$.
    Thus, $\pi_x$ and $\pi_{x+1}$ are in distinct blocks of $\mathcal{B}$. 
    
    By \Cref{thm: OSP preimage = flatten preimage}, since $\Flatten(\mathcal{B})=\out_{\FR{n}}(\alpha)=\pi=\pi_1\pi_2\cdots\pi_n$ and since $\pi_x$ and $\pi_{x+1}$ appear consecutively in $\pi$, we have that $\pi_x$ and $\pi_{x+1}$ are in adjacent blocks $B_i$ and $B_{i+1}$ of $\mathcal{B}$ (for some $i\in[t-1]$). 
    Then, since the elements in $B_i$ and $B_{i+1}$ are sorted in increasing order and $\pi_x$ and $\pi_{x+1}$ appear consecutively in $\pi$, we have that $\pi_x$ must be the last value (i.e., maximum) in $B_i$ and $\pi_{x+1}$ must be the first value (i.e., minimum) in $B_{i+1}$, as claimed.
\end{proof}

We now count the number of ordered set partitions that flatten to a fixed permutation.
\begin{corollary}\label{corr: count of flat(pi) is 2^n-k}
 If $\pi\in\Sym_n$ has $k$ runs, then $|\Flatten^{-1}(\pi)|=2^{n-k}$.
\end{corollary}
\begin{proof}
    Given $\pi\in\Sym_{n}$ with runs $\pi=\tau_1\tau_2\cdots\tau_k$, construct the ordered set partition $\mathcal{B}=(B_1,B_2,\ldots,B_k)$ where, for each $i\in[k]$, $B_i$ consists of the elements in $\tau_i$ ordered in increasing order.
    By \Cref{thm:ordered set partition is the outcome}, $\Flatten(\mathcal{B})=\pi$.
    
    First we show that partitioning a block, and keeping the relative order of the elements in that block, results in a ordered set partition whose flattened version is exactly $\pi$.
    Consider a block $B_i=\{\tau_{i,1},\tau_{i,2},\ldots,\tau_{i,|\len(\tau_i)|}\}$ where $\tau_{i,j}$ denotes the $j$-th entry in run $\tau_i$.
    Replacing any comma in $B_i$ with ``$\},\{$'' (while keeping the elements in the same order as in $B_i$)
    results in a new ordered collection of blocks we call $B_i'$.
    Replacing $B_i$ in $\mathcal{B}$ with the collection of ordered sets $B_i'$, yields an ordered set partition whose flattened parking outcome is $\pi$. 
    Furthermore, any subset of commas can be replaced with instances of ``$\},\{$'' and still yield an ordered set partition which can be flattened to $\pi$.
    Therefore, since each $B_i$ has $\len(\tau_i)-1$ commas, 
    we can construct $|\mathcal{P}([\len(\tau_i)-1])| = 2^{\len(\tau_i)-1}$ distinct ordered set partitions in $\Flatten^{-1}(\pi)$ by only changing $B_i$, where $\mathcal{P}(X)$ denotes the power set for a set $X$.
    Hence, for each run $\tau_i$ we can construct $2^{\len(\tau_i)-1}$ ordered set partitions and, taking the product over each $i\in[k]$, we get a total of 
    \[\prod_{i=1}^k2^{\len(\tau_i)-1}=2^{n-k}\]
    ordered set partitions that flatten to $\pi$. 
    Call this collection of constructed ordered set partitions $\mathfrak{B}\subseteq \Flatten^{-1}(\pi)$. Then $|\mathfrak{B}|=2^{n-k}$, and so
    \[|\Flatten^{-1}(\pi)|\geq 2^{n-k}.\]

    To show equality, we now establish that any ordered set partition that flattens to $\pi$ must have come from our construction described above, which means that it had to be in $\mathfrak{B}$ to begin with.
    
    Consider $\mathcal{A}=(A_1,A_2,\ldots,A_s)$ with  $\Flatten(\mathcal{A})=\pi$.
    The values in $[n]$ appear in 
    $\mathcal{A}$ exactly in the order they appear in $\pi$.
    Moreover, 
    by \Cref{lem:descent force blocks} every descent in $\pi$ forces the start of a new block in $\mathcal{A}$.
    By assumption,
    $\pi$ has $k$ runs, which means $\pi$ has $k-1$ descents.
    Hence, $\mathcal{A}$ has at least $k$ blocks: whenever $\pi_i>\pi_{i+1}$, a block ends on $\pi_i$ and a new block starts on $\pi_{i+1}$, while the last block ends on $\pi_n$.
    If these are the only blocks of $\mathcal{A}$, then $\mathcal{A}=(\tau_1,\tau_2,\ldots,\tau_k)$ and we are done.
    If $\mathcal{A}$ has more blocks, these blocks would further partition the blocks made up of the runs of $\pi$, which is precisely how we constructed the elements in the set $\mathfrak{B}$.
    Therefore $\mathcal{A}\in\mathfrak{B}$.
    This proves that 
    \[|\Flatten^{-1}(\pi)|=2^{n-k},\]
    where $k$ is the number of runs of $\pi$.
\end{proof}

We now define the restriction of a Fubini ranking with a given parking outcome $\pi$ to a run of $\pi$.
\begin{definition}\label{def:restrict alpha to a run}
    Fix $\pi\in\Sym_n$ and let $\tau=\tau_1\tau_2\cdots\tau_k$ be a run of $\pi$ with length $k$.
If $\alpha=(a_1,a_2,\ldots,a_n)\in\FR{n}$ such that $\out_{\FR{n}}(\alpha)=\pi$, then \textit{the restriction of $\alpha$ to the competitors in $\tau$}, is $\alpha|_{\tau}=(a_{\tau_1},a_{\tau_2},\ldots,a_{\tau_{k}})$.
Then the collection of all Fubini rankings with outcome $\pi$ restricted to the competitors in a run $\tau$ is 
\[\FR{\tau}=\{\alpha|_{\tau}:\alpha\in\out_{\FR{n}}^{-1}
(\pi)\}\subseteq[n]^k.\]
\end{definition}
For example, 
if $\pi=1324$ and $\tau=24$, then 
$\out_{\FR{4}}^{-1}(1324)=\{(1,3,1,3),(1,3,1,4),(1,3,2,3),(1,3,2,4)\}$,
and so $\FR{\tau}=\{(3,3),(3,4)\}$.
We now introduce a technical lemma that is  helpful in our next set of results.

\begin{lemma}\label{lem:save us}
Fix $\pi\in\Sym_n$ and let $\tau$ be a run of $\pi$. Then the sets $\FR{\tau}$ and $\FR{\len(\tau)}^{\uparrow}$ are in bijection.
\end{lemma}
\begin{proof}
    Let $k = \len(\tau)$. We define the function $\varphi : \FR{\tau} \to \FR{k}^{\uparrow}$, with $\beta=(b_1,b_2,\ldots,b_k)\in \FR{\tau}$, by
    \[\varphi(\beta) = \varphi((b_1, b_2, \ldots, b_k)) = (b_1 - (b_1-1), b_2 - (b_1-1), \ldots, b_k - (b_1-1)).\]
    We claim that $\varphi$ is injective and surjective, and hence a bijection. 
    Before that, we begin by showing that $\varphi$ is well-defined: namely, that for any $\alpha \in \FR{\tau}$, we have $\varphi(\alpha) \in \FR{k}^{\uparrow}$.\\
    
    \noindent\textbf{Well-defined.} 
    Suppose
    that the run $\tau$ of $\pi$ has the form $\tau= \tau_1 \tau_2, \cdots \tau_k$. 
        Let $\beta\in\FR{\tau}$. 
        By \Cref{def:restrict alpha to a run}, there exists
    $\alpha=(a_1,a_2,\ldots,a_n)\in\out_{\FR{n}}^{-1}(\pi)$ such that $\beta=\alpha|_{\tau}$.
    In the context of Fubini rankings, recall that for each $i\in[k]$, $\tau_i$ is a competitor in the race with rank $a_{\tau_i}$.
    Since  $\tau$ is a run of $\pi$, we claim that the competitors $\tau_1, \tau_2, \ldots, \tau_k$ appear in consecutive indices in $\pi$ and the ranks of these competitors must be weakly increasing.
    Recall that the ordered set partition corresponding to $\alpha$ is $\psi(\alpha)=(B_1,B_2,\ldots,B_t)$ where for each $j\in[t]$, $B_j=\{a_i:a_i=j\}$, and the entries in $B_j$ are written in increasing order. Moreover, since $\Flatten(\psi(\alpha))=\pi$ (see \Cref{thm:ordered set partition is the outcome}), then for a specific run $\tau=\tau_1\tau_2\cdots\tau_k$, $\tau_1$ corresponds to the first entry $B_{(x,1)}$ in some block $B_x$, and $\tau_k$ corresponds to the last entry of a block $B_{(y,\len(B_y))}$ for some $x\leq y$. 
    Additionally, every entry between $B_{(x,1)}$ and $B_{(y,\len(B_y))}$ is contained in the run $\tau$, otherwise there would exist some descent between $B_{(x,1)}$ and $B_{(y,\len(B_y))}$, which is a contradiction to $\tau$ being a run of ascents in $\pi$. Therefore, the ranks of the competitors in $\tau$ appear in weakly increasing order in $\beta=\alpha|_\tau$.
    Thus $\beta=\alpha|_\tau=(a_{\tau_1},a_{\tau_2},\ldots,a_{\tau_k})$ satisfies $a_{\tau_i}\leq a_{\tau_{i+1}}$ for all $i\in[k-1]$.
    Therefore, $\beta = (a_{\tau_1}, a_{\tau_2}, \ldots, a_{\tau_k})$ is a \textit{subtuple} of $\alpha^{\uparrow}=(a_1',a_2',\ldots,a_n')$, which is the weakly increasing rearrangement of $\alpha$.  
    By subtracting $a_{\tau_1}-1$ from each entry in $\beta$,
    we obtain the tuple 
    \[(a_{\tau_1}-(a_{\tau_1}-1),a_{\tau_2}-(a_{\tau_1}-1),\ldots,a_{\tau_k}-(a_{\tau_1}-1))= \varphi(\beta).\] 
    Note that $a_{\tau_i}-(a_{\tau_1}-1)\in [n]$ for all $i\in[k]$, and hence $\varphi(\beta) \in [n]^k$.
    Moreover,  $a_{\tau_1}-(a_{\tau_1}-1)=1$. 
    Since $\beta$
    is weakly increasing, then $\varphi(\beta)$ is weakly increasing, so it suffices to 
    show that $\varphi(\beta)$ is a Fubini ranking with $k$ competitors.

    Since $\tau$ is a run of $\pi$ (which is a maximal contiguous increasing subword), and $\alpha|_{\tau}=(a_{\tau_1}, a_{\tau_2}, \ldots , a_{\tau_k})$ is weakly increasing, 
    then $\alpha|_{\tau}$ appears in $\alpha^{\uparrow}=(a_1',a_2',\ldots,a_n')$ with $a_i\leq a_{i+1}'$ for all $i\in[n]$, in consecutive entries. 
    Thus, there exists an integer $r \in [0,n-1]$, such that for every $i \in [k]$, $a_{\tau_i} = a'_{i+r}$. \\
    
    \noindent \textbf{Claim 1}:   $a_{r}'<a_{r+1}'$.\\
    
    Suppose for contradiction that $a_r'=a_{r+1}'$. 
    Then there exists a competitor $x$ with rank $a_{r}'=a_{r+1}'=a_{\tau_1}$. 
    There are now two cases to consider: $x\in\tau$ or $x\notin\tau$. 
    \begin{itemize}
    \item If $x\in\tau$, then $x>\tau_1$, since $\tau_1$ is the first element of $\tau$.
However, the competitors in $\tau$ have preferences $a_{i+r}'$ with $i\in[k]$, which implies that we now have an additional competitor $x$ in the run $\tau$ besides the competitors $\tau_1,\tau_2,\ldots,\tau_k$, and so $\len(\tau)>k$, a contradiction. Thus, $x\notin\tau$. 
\item If $x\notin\tau$, then by \Cref{lemma: ties do not break over runs}, $a_{r}'\neq a_{r+1}'$, as desired.
    \end{itemize}
    Thus, we have established the claim.\\

\noindent \textbf{Claim 2}: $\varphi(\beta)=(a_{\tau_1}-(a_{\tau_1}-1),a_{\tau_2}-(a_{\tau_1}-1),\ldots,a_{\tau_k}-(a_{\tau_1}-1))$ satisfies the conditions in \Cref{lemma: weakly increasing Fubini proof}. \\

To prove this claim, we recall that since $\alpha$ is a Fubini ranking, $\alpha^{\uparrow}$ satisfies the conditions in \Cref{lemma: weakly increasing Fubini proof}. Namely, for every $i \in [n]$, we have $a'_i = a'_{i-1}$ or $a'_i = i$. 
    By Claim 1, since  $a'_{r} < a'_{r+1}$, we must have that $a'_{r+1} = r+1$. 
    This implies that
    \begin{align}\label{eq:solve for r}r+1 = a'_{r+1} = a_{\tau_1}. \end{align}
    Solving for $r$ in \Cref{eq:solve for r} yields $r = a_{\tau_1}-1$. 
    Moreover, for any $i \in [k]$, we can use the facts that $a_{\tau_i}=a_{i+r}'$ and $r=a_{\tau_1}-1$, to arrive at the following:
    \begin{itemize}
        \item If $a'_{i+r}=a'_{(i-1)+r}$, then  the $i$th entry of $\varphi(\beta)$ satisfies
        \begin{align*}
            a_{\tau_i}-(a_{\tau_1}-1) &= a'_{r+i}-r &\mbox{(since $a_{\tau_i}=a_{i+r}'$ and $a_{\tau_1}-1=r$)}\\
            &= a'_{(i-1)+r}-r &\mbox{(by assumption $a_{(i-1)+r}'=a_{i+r}'$)}\\
            &= a_{\tau_{i-1}}-(a_{\tau_1}-1) &\mbox{(since $a_{\tau_{i-1}}=a_{(i-1)+r}'$)}.
        \end{align*}
         This implies $a_{\tau_i}=a_{\tau_{i-1}}$.
        
        \item If $a'_{i+r}=i+r$, then the $i$th entry of $\varphi(\beta)$ satisfies
        \begin{align*}
        a_{\tau_i} -(a_{\tau_1}-1)&= a'_{i+r}-r &\mbox{(since $a_{\tau_i}=a_{i+r}'$ and $a_{\tau_1}-1=r$)}\\
        &= i+r-r &\mbox{(by assumption $a_{i+r}'=i+r$)}\\
        &= i    .
        \end{align*}
        This implies that $a_{\tau_i}=i$.
    \end{itemize}
Therefore, $\varphi(\beta)$ satisfies the conditions in \Cref{lemma: weakly increasing Fubini proof}, and thus $\varphi(\beta)\in\FR{k}^\uparrow$, as claimed.
\\

\noindent\textbf{Injectivity:} Let $\alpha \in \FR{n}$, and let $\mathcal{B} = \psi(\alpha)$, with $\psi$ defined in \Cref{lem:bijection between Fubini rankings and ordered set partitions}. Recall that $\tau = \tau_1 \tau_2\cdots \tau_k$ is a run in $\pi$. 
    By \Cref{lem:descent force blocks}, we know that $\tau_1$ is the first element in some block $B_m$. By \Cref{lemma: ranks determined by number of tied competitors}, we know that the rank of competitor $\tau_1$ is $1 + \sum_{i=1}^{m-1} |B_i|$. By \Cref{thm:ordered set partition is the outcome}, we have that $1 + \sum_{i=1}^{m-1} |B_i|$ also happens to be the index of $\tau_i$ in $\pi$. 
    Since the index of $\tau_1$ in $\pi$ depends only on $\pi$, the rank of competitor $\tau_1$ is the same in any Fubini ranking $\alpha \in \out^{-1}_{\FR{n}}(\pi)$.  

    Let $\beta=(b_1,_2,\ldots,b_k), \delta=(d_1,d_2,\ldots,d_k) \in \FR{\tau}$ and suppose that $\varphi(\beta) = \varphi(\delta)$.  Hence $(b_1-(b_1-1), b_2-(b_1-1), \ldots, b_k - (b_1-1)) = (d_1-(d_1-1), d_2-(d_1-1), \ldots, d_k - (d_1-1))$. However, by definition of $\FR{\tau}$, we have that $b_1 = d_1 = a_{\tau_1}$, the rank of competitor $\tau_1$ in \textit{any} $\alpha \in \FR{n}$. 
    Therefore, applying $\varphi$ to $\beta$ and $\delta$ causes the same number to be subtracted from their entries. Since $\varphi(\beta) = \varphi(\delta)$, we must then also have $\beta = \delta$. Thus $\varphi$ is injective.\\ 
    
\noindent \textbf{Cardinality of the sets:} Since the sets $\FR{\tau}$ and $\FR{k}^\uparrow$ 
are finite sets, and $\varphi$ is injective, the bijection follows provided the sets have the same cardinality.
By \Cref{lem:weakly increasing Fubinis park in identity order}, $|\FR{k}^\uparrow|=2^{k-1}$.
Now consider $\FR{\tau}$, which is the set of Fubini rankings on $n$ competitors  restricted to the indices in $\tau$. 
Consider $\alpha\in\out_{\FR{n}}^{-1}(\pi)$. 
Using the bijection $\psi$ (\Cref{lem:bijection between Fubini rankings and ordered set partitions}), the element $\alpha$ has a corresponding ordered set partition of $[n]$, i.e.,
$\psi(\alpha)=(B_1,B_2,\ldots,B_t)$. 
Since $\tau$ is a run, i.e., a maximal contiguous increasing subword of $\pi$ 
, by \Cref{thm:ordered set partition is the outcome}, 
there exist indices $x\leq y$  such that 
$\tau=\Flatten(B_x,B_{x+1},\ldots,B_y)$. 
Moreover, the number of values appearing in $\tau$ is $k=\len(\tau)$, and so the number of elements appearing in $\Flatten(B_x,B_{x+1},\ldots,B_y)$ is also $k$. Then we can further subdivide the set partition $(B_x,B_{x+1},\ldots,B_y)$ into other ordered set partitions by taking a comma between any two elements in any block and replacing the comma with ``$\},\{$''. 
Any ordered set partition constructed in this way flattens to $\tau$.
Since there are $k-1$ commas in the set partition $(B_x,B_{x+1},\ldots,B_y)$ that we can replace with ``$\},\{$'', we can construct $2^{k-1}$ many distinct ordered set partitions, all of which flatten to $\tau$. 
Each such ordered set partition corresponds uniquely to an element of $\FR{\tau}$. 
Using an analogous argument as that presented in the proof of \Cref
{corr: count of flat(pi) is 2^n-k}, it follows that every element of $\FR{\tau}$ can be constructed as described. 
Therefore, $|\FR{\tau}|=2^{k-1}$, which completes the proof.
\qedhere
    
\end{proof}
\begin{remark}\label{remark: a normalized run is a fubini ranking}
    \Cref{lem:save us} shows that for a given $\pi\in\Sym_n$, with $\pi$ consisting of $k$ runs, we can count the set $\out^{-1}_{\FR{n}}(\pi)$ by counting separately the number of ways to assign valid ranks to the competitors in the runs of $\pi$. Moreover, any tuple of valid ranks for the competitors of a particular run $\tau$ that is standardized to be a tuple in $[\len(\tau)]$ is a valid Fubini ranking of length $\len(\tau)$. Therefore, by \Cref{lem:save us},
    \begin{equation*}
        |\out^{-1}_{\FR{n}}(\pi) |= \prod_{i=1}^k |\out^{-1}_{\FR{\len(\tau_i)}}(123\cdots \len(\tau_i))|.
    \end{equation*}
\end{remark}
\begin{example}
    Consider $\pi\in\Sym_8$ such that $\pi=36145278$. Note that the runs of $\pi$ are $\tau_1=36$ with $\len(\tau_1)=2$, $\tau_2=145$ with $\len(\tau_2)=3$, and $\tau_3=278$ with $\len(\tau_3)=3$. There are 32 Fubini rankings with outcome $\pi$:
     \resizebox{\textwidth}{!}{
        $\out^{-1}_{\FR{n}}(\pi)=
        \left\{
        \begin{tabular}{ccccc}
        (3, 6, 1, 3, 3, 1, 6, 6),&
        (3, 6, 1, 3, 3, 1, 6, 8),&
        (3, 6, 1, 3, 3, 1, 7, 7),&
        (3, 6, 1, 3, 3, 1, 7, 8),&
        (3, 6, 1, 3, 5, 1, 6, 6), \\
        (3, 6, 1, 3, 5, 1, 6, 8),&
        (3, 6, 1, 3, 5, 1, 7, 7),&
        (3, 6, 1, 3, 5, 1, 7, 8),&
        (3, 6, 1, 4, 4, 1, 6, 6),&
        (3, 6, 1, 4, 4, 1, 6, 8),\\
        (3, 6, 1, 4, 4, 1, 7, 7),&
        (3, 6, 1, 4, 4, 1, 7, 8),&
        (3, 6, 1, 4, 5, 1, 6, 6),&
        (3, 6, 1, 4, 5, 1, 6, 8),&
        (3, 6, 1, 4, 5, 1, 7, 7), \\
        (3, 6, 1, 4, 5, 1, 7, 8),&
        (3, 6, 1, 3, 3, 2, 6, 6),&
        (3, 6, 1, 3, 3, 2, 6, 8),&
        (3, 6, 1, 3, 3, 2, 7, 7),&
        (3, 6, 1, 3, 3, 2, 7, 8), \\
        (3, 6, 1, 3, 5, 2, 6, 6),&
        (3, 6, 1, 3, 5, 2, 6, 8),&
        (3, 6, 1, 3, 5, 2, 7, 7),&
        (3, 6, 1, 3, 5, 2, 7, 8),&
        (3, 6, 1, 4, 4, 2, 6, 6), \\
        (3, 6, 1, 4, 4, 2, 6, 8),&
        (3, 6, 1, 4, 4, 2, 7, 7),&
        (3, 6, 1, 4, 4, 2, 7, 8),&
        (3, 6, 1, 4, 5, 2, 6, 6),&
        (3, 6, 1, 4, 5, 2, 6, 8), \\
        (3, 6, 1, 4, 5, 2, 7, 7),&
        (3, 6, 1, 4, 5, 2, 7, 8)&&&\\
    \end{tabular}
    \right\}.$}
    
    \noindent Moreover, 
    \begin{align*}
        \out^{-1}_{\FR{n}}(12) = \{(1,1),(1,2)\} \quad \text{ and } \quad \out^{-1}_{\FR{n}}(123) = \{(1,1,1), (1,1,3), (1,2,2), (1,2,3)\}. 
    \end{align*}
    Notice that 
    \[32 = 2^5 = 2^{\len(\tau_1)-1}\cdot 2^{\len(\tau_2)-1}\cdot 2^{\len(\tau_3)-1},\]
    and by \Cref{lem:weakly increasing Fubinis park in identity order},
    \[|\out^{-1}_{\FR{n}}(123\cdots n )|=2^{n-1},\]
    which implies 
    \[|\out^{-1}_{\FR{n}}(\pi )|=\prod_{i=1}^k 2^{\len(\tau_i)-1}=\prod_{i=1}^k |\out^{-1}_{\FR{\len(\tau_i)}}(123\cdots \len(\tau_i))|.\]
\end{example}
We now give the Fubini ranking analog to \Cref{corr: count of flat(pi) is 2^n-k}. 

\begin{corollary}\label{cor: fubini count is 2^n-k}
    If $\pi\in\Sym_n$ has $k$ runs,  then
    $|\out_{\FR{n}}^{-1}(\pi)|=2^{n-k}$. 
    Moreover, if $\tau_1,\tau_2,\ldots,\tau_k$ are the runs of $\pi$, then 
    $|\out_{\FR{n}}^{-1}(\pi)|=\prod_{i=1}^k|\out_{\FR{\len(\tau_i)}}^{-1}(123\cdots\len(\tau_i))|$.
\end{corollary}
\begin{proof}
    If $\pi\in\Sym_n$ has $k$ runs, then by \Cref{thm: OSP preimage = flatten preimage} we have $|\psi(\out_{\FR{n}}^{-1}(\pi))|=|\Flatten^{-1}(\pi)|$, and by \Cref{corr: count of flat(pi) is 2^n-k} we have $|\Flatten^{-1}(\pi)|=2^{n-k}$, so by transitivity
    $|\psi(\out_{\FR{n}}^{-1}(\pi))|=2^{n-k}$.

    For the second statement,    consider $\pi\in\Sym_{n}$ with runs $\tau_1,\tau_2,\ldots,\tau_k$. Construct the ordered set partition $\mathcal{B}=(B_1,B_2,\ldots,B_k)$, where, for each $i\in[k]$, $B_i$ consists of the elements in $\tau_i$ listed in increasing order: $\tau_{i,1},\tau_{i,2},\ldots,\tau_{i,\len(\tau_i)}$. 
    Then $\Flatten(\mathcal{B})=\pi$, and in particular $\Flatten(B_i)=\tau_i$ for each $i\in[k]$.
 In the proof of \Cref{corr: count of flat(pi) is 2^n-k}, we established that for each run $\tau_i$ we can construct $2^{\len(\tau_i)-1}$ distinct ordered set partitions of the block $B_i$, whose flattening recovers the run $\tau_i$ in $\pi$. 
 Under the bijection $\psi$, this implies that for each $i\in[k]$, there are $2^{\len(\tau_i)-1}$ distinct ways to place values at the indices in $B_i$ (of a tuple of length $n$) to construct a Fubini ranking whose outcome recovers the run $\tau_i$. 
 Standardize each run $\tau_i$ to the identity permutation $123\cdots \len(\tau_i)\in \Sym_{\len(\tau_i)}$.
 By \Cref{lem:weakly increasing Fubinis park in identity order}, we know that
 for any $m\in \mathbb{N}$, $|\out_{\FR{m}}^{-1}(123\cdots m)|=2^{m-1}$. 
 Taking the product over all $i\in[k]$,
 we have that  
 \[|\out_{\FR{n}}^{-1}(\pi)|=\prod_{i=1}^k |\out_{\FR{\len(\tau_i)}}^{-1}(123\cdots \len(\tau_i))|=\prod_{i=1}^k2^{\len(\tau_i)-1}=2^{n-k}.\] This proves the second statement, while also giving an alternate proof of the first statement.
\end{proof}

We can now use our previous results to give the following formula for the Fubini numbers by counting Fubini rankings based on their parking outcome.

\begin{corollary}\label{corr: fubini numbers are odd}
    For $n\in\NN$, the Fubini numbers are enumerated by 
    \begin{equation}\label{eq: fubini is 1 + stuff}
        \Fub_n=\sum_{\pi\in\Sym_n}|\out_{\FR{n}}^{-1}(\pi)|=\sum_{k=1}^{n}2^{n-k}\cdot A(n,k-1),
    \end{equation}
    where $A(n,k-1)$ is the number of permutations with $k$ runs.
\end{corollary}
\begin{proof}
    If $n=1$, then there is only one Fubini ranking $\alpha=(1)$, so $\Fub_1=1$. 
    Moreover, the right-hand side of \eqref{eq: fubini is 1 + stuff} is
    \begin{equation}
        \sum_{k=1}^1 2^{1-k}\cdot A(1,k-1) = 2^0\cdot A(1,0) = 1.
    \end{equation}
    Next, fix $n\in \NN$ such that $n>1$. Since every Fubini ranking has a unique parking outcome, we count the number of Fubini rankings of length $n$ using the possible parking outcomes, i.e., 
    \begin{equation}\label{eq: number of Fub's is sum of outcomes}
    \Fub_n=\sum_{\pi\in\mathfrak{S}_n}\left|\out_{\FR{n}}^{-1}(\pi)\right|. 
    \end{equation}
    Choose an arbitrary $\pi\in\mathfrak{S}_n$. By \Cref{corr: count of flat(pi) is 2^n-k}, the number of Fubini rankings with parking outcome $\pi$ is $2^{n-k}$, where $k$ is the number of runs in $\pi$. 
    Therefore, we can rewrite the equality as \begin{equation}\label{eq: number of Fub's is sum of powers of two}
        \Fub_n=\sum_{k=1}^{n} \sum_{\substack{\pi\in\mathfrak{S}_n \\ \pi = \tau_1\tau_2\cdots\tau_k}}2^{n-k}. 
    \end{equation}
    Furthermore, if $\pi$ is a product of $k$ runs then it contains $k-1$ descents. The number of $\pi\in\mathfrak{S}_n$ with $k-1$ descents is enumerated by the Eulerian number $A(n,k-1)$. Therefore, \Cref{eq: number of Fub's is sum of outcomes} can be expressed as
    \begin{equation}\label{eq: the punchline}
        \Fub_n= \sum_{\pi\in\mathfrak{S}_n}\left|\out_{\FR{n}}^{-1}(\pi)\right| = \sum_{k=1}^{n}2^{n-k}\cdot A(n,k-1),
    \end{equation}
    which completes the proof.
\end{proof}

\section{Enumerating Unit Fubini Rankings by their Parking Outcome}\label{sec:UFRs}
A unit Fubini ranking is a Fubini ranking in which at most two competitors tie for any single rank. 
For example,  $(1,1,3)$ is a unit Fubini ranking while $(1,1,1)$ is not, as three competitors tied for the first rank. 
We let $\UFR{n}$ denote the set of unit Fubini rankings with $n$ competitors. 
It is known that $|\UFR{n}|$ satisfies the recurrence \cite[\seqnum{A080599}]{OEIS}: 
\[|\UFR{n}|=n\cdot|\UFR{n-1}|+n\cdot\left(\frac{n-1}{2}\right)\cdot |\UFR{n-2}|\]
and a closed formula is given by 
\[|\UFR{n}|=n!\frac{\left((1+\sqrt{3})^{n+1}-(1-\sqrt{3})^{n+1}\right)}{2^{n+1}\sqrt{3}}.\]

As is our convention, we let $\UFR{n}^\uparrow$ denote the set of weakly increasing unit Fubini rankings of length $n$. 
In \cite[Theorem 3.18]{weak}, Elder et.\ al prove that the number of weakly increasing unit Fubini rankings of length $n$ is enumerated by the $n$-th Fibonacci number, i.e., $|\UFR{n}^\uparrow|=F_{n}$. We recall that the Fibonacci numbers satisfy the recurrence,
\[F_n=F_{n-1}+F_{n-2}.\]
For convenience, we set the starting values at $F_1=1$ and $F_2=2$, so that the index on the length of the Fubini rankings agrees with the index of the Fibonacci numbers. 
This result also follows from 
restricting \Cref{lem:weakly increasing Fubinis park in identity order} to the set of unit Fubini rankings, which implies that 
\[\out_{\UFR{n}}^{-1}(123\cdots n)=\UFR{n}^\uparrow.\]
Hence the cardinalities of the sets are also equal. 
We state this formally below.

 \begin{corollary}\cite[Theorem 3.18]{weak}:\label{cor:outcome identity is all weakly increasing}
     For all $n\in\mathbb{N}$, $|\out^{-1}_{\UFR{n}}(123\cdots n)|=|\UFR{n}^\uparrow|=F_{n}$, where $F_{n}$ is the $n$-th Fibonacci number.
 \end{corollary}

We now use \Cref{cor:outcome identity is all weakly increasing} to give counts for the number of unit Fubini rankings with a fixed parking outcome.

\begin{theorem}\label{thm:number of UFR with outcome pi}
        If $\pi\in\mathfrak{S}_n$ has runs  $\tau_1,\tau_2,\ldots,\tau_k$ with lengths $\len(\tau_1), \len(\tau_2), \cdots, \len(\tau_{k})$, respectively, 
        then 
    \begin{equation}
        |\out^{-1}_{\UFR{n}}(\pi)| 
        = \prod_{i=1}^kF_{\len(\tau_i)}\qedhere
    \end{equation}
    where $F_i$ denotes the $i$-th Fibonacci number.
\end{theorem}
\begin{proof}
      By \Cref{cor:outcome identity is all weakly increasing}, if $\pi = 1 2 \cdots n$, then $|\out^{-1}_{\UFR{n}}(\pi)| = |\UFR{n}^{\uparrow}| = F_n$.
      Now suppose that $\pi$ has runs $\tau_1,\tau_2,\ldots,\tau_k$. Construct the ordered set partition $\mathcal{B} = (B_1,B_2,\ldots, B_k)$, where, for each $i\in[k]$, $B_i$ consists of the elements in $\tau_i$ listed in increasing order: $\tau_{i,1}, \tau_{i,2},\ldots, \tau_{i,\len(\tau_i)}$. Then $\Flatten(\mathcal{B})=\pi$, and in particular $\Flatten(B_i)=\tau_i$ for each $i\in [k]$.
      Using an analogous argument as that in \Cref{cor: fubini count is 2^n-k} for weakly increasing unit Fubini rankings, we get that 
       \[|\out_{\UFR{n}}^{-1}(\pi)|=\prod_{i=1}^k |\out_{\UFR{\len(\tau_i)}}^{-1}(123\cdots \len(\tau_i))|=\prod_{i=1}^kF_{\len(\tau_i)},\]
      which gives the desired count.
     \qedhere
\end{proof}

We now give a new formula for the number of unit Fubini rankings with $n$ competitors.
\begin{corollary}\label{cor:formula for uni fubini using compositions}
Fix $n\in \NN$. 
Let $\Comp(n,k)$ be the set of compositions of $n$ with $k$ parts. Given $\textbf{c}=(c_1,c_2,\ldots,c_k)\in\Comp(n,k)$, let $A(n,\textbf{c})$ denote the number of permutations $\pi=\tau_1\tau_2\cdots\tau_k\in\Sym_n$ with $k$ runs which satisfy $\len(\tau_i)=c_i$ for all $i\in[k]$. 
Then 
    \[|\UFR{n}| =\sum_{k=1}^{n}\left(\sum_{\textbf{c}=(c_1,c_2,\ldots,c_k)\in\Comp(n,k)}A(n,\textbf{c})\cdot\prod_{i=1}^kF_{c_i}\right).\]
\end{corollary}
\begin{proof}
By definition \[|\UFR{n}|=\sum_{\pi\in\Sym_n}|\out_{\UFR{n}}^{-1}(\pi)|.\]
By \Cref{thm:number of UFR with outcome pi}, we know $|\out_{\UFR{n}}^{-1}(\pi)|=\prod_{i=1}^k F_{c_i}$. 
Hence, given $\textbf{c}=(c_1,c_2,\ldots,c_k)\in\Comp(n,k)$,  there are $A(n,\textbf{c})$ permutations $\pi$ with runs of lengths given by $\textbf{c}$, and each of those permutations has the same number of unit Fubini rankings parking in that order. 
Thus, we must account for all possible compositions of run lengths and we must also vary the number of runs from 1 to $n$, which yields
\begin{align*}
|\UFR{n}|&=\sum_{\pi\in\Sym_n}|\out_{\UFR{n}}^{-1}(\pi)|\\
&=\sum_{k=1}^{n}\left(\sum_{\textbf{c}=(c_1,c_2,\ldots,c_k)\in\Comp(n,k)}A(n,\textbf{c})\cdot\prod_{i=1}^kF_{c_i}\right)
\end{align*}
as claimed, and this completes the proof.
\end{proof}

\begin{remark}
 Given a composition $\textbf{c}=(c_1,c_2,\ldots,c_k)\in\Comp(n,k)$, the function $A(n,\textbf{c})$ equivalently counts the number of permutations with descents at indices 
 \[i_1=c_1, \quad  i_2=c_1+c_2, \quad \ldots, \quad i_{k-1}= c_1+c_2+\cdots+c_{k-1}.\]
 To determine the number of permutations with a given descent set one can use the recurrences in
 \cite[Proposition 2.1 or Theorem 2.4]{descentpolys}.
 Those formulas then allow us to compute the value of $A(n,\textbf{c})$.
\end{remark}

\begin{example}
One can verify that 
\begin{align*}
    &\UFR{3}\\
    &\quad =\{(1,1,3),(1,3,1),(3,1,1),(1,2,2),(2,1,2),(2,2,1),(1,2,3),(1,3,2),(2,1,3),(2,3,1),(3,1,2),(3,2,1)\}.
\end{align*} 
Hence $|\UFR{3}|=12$.
When $n=3$, observe that 
\begin{itemize}
    \item the only permutation with a single run of length $3$ is the identity permutation $123$, so $A(3,(3))=1$,
    \item  there are two permutations with runs of lengths $(2,1)$ (they are $132,231$), so $A(3,(2,1))=2$, 
    
    \item there are two permutations with runs of lengths $(1,2)$ (they are $213,312$), so $A(3,(1,2))=2$, and
    \item the only permutation of length three with runs of lengths $(1,1,1)$ is $321$, so $A(3,(1,1,1))=1$.
\end{itemize}
Using \Cref{cor:formula for uni fubini using compositions} and setting $n=3$ yields
\begin{align*}
\sum_{\substack{\textbf{c}=(c_1) \\ \textbf{c} \in\Comp(3,1)}}&A(n,\textbf{c})\cdot F_{c_1}+
\sum_{\substack{\textbf{c}=(c_1,c_2) \\ \textbf{c} \in\Comp(3,2)}}A(n,\textbf{c})\cdot F_{c_1}\cdot F_{c_2}+
\sum_{\substack{\textbf{c}=(c_1,c_2,c_3) \\ \textbf{c} \in\Comp(3,3)}}A(n,\textbf{c})\cdot F_{c_1}\cdot F_{c_2}\cdot F_{c_3}\\
&=A(3,(3))F_3+[A(3,(2,1))F_2F_1+A(3,(1,2))F_1F_2]+A(3,(1,1,1))F_1F_1F_1\\
&=1\cdot 3+[2\cdot 2\cdot 1+2\cdot 1\cdot 2]+1\cdot 1\cdot1\cdot1\\
&=12=|\UFR{3}|,
\end{align*}
as expected.    
\end{example}

\section{Enumerating $\ell$-interval Fubini Rankings by their Outcome}\label{sec:ell FRs}
We now consider the generalization of unit Fubini rankings that allow for at most $\ell+1$
competitors to tie at a single rank.
Fubini rankings with this property are called $\ell$-interval
Fubini rankings, and we denote the set of all $\ell$-interval Fubini rankings of length $n$ by $\FR{n}(\ell)$. 
For example, the set $\FR{n}(0)$ is the set of permutations of $[n]$, the set $\FR{n}(1)$ is the set of unit Fubini rankings, and for all $\ell\geq n-1$ the set $\FR{n}(\ell)$ is the set of all Fubini rankings with $n$ competitors. 
In fact, by definition, 
$\FR{n}(\ell)\subseteq \FR{n}(\ell+1)$ for all $\ell\in\NN$.
As before, we let $\FR{n}^\uparrow(\ell)$ denote the set of $\ell$-interval
Fubini ranking that are weakly increasing. 
The number of  weakly increasing $\ell$-interval Fubini rankings can be determined via the following recursive formula.
 
\begin{theorem}[Theorem 4.16 in \cite{aguilarfraga2024interval}]\label{thm:recurrences for weakly increasing ell frs}
    If $n,\ell\in\NN$ with $\ell+1\leq n$, then 
    \begin{equation}   \left|\FR{n}^\uparrow(\ell)\right| = \sum_{x=0}^\ell \left|\FR{n-1-x}^\uparrow(\ell)\right|, \label{eq: count for weakly increasing Fubini}
    \end{equation}
    where $\left|\FR{0}(\ell)\right|=1$.
\end{theorem}

The recurrences defined in \Cref{thm:recurrences for weakly increasing ell frs} are generalizations of the Fibonacci numbers and we define those next.

\begin{definition}\label{def:ell-Pingala}
    Let $\ell\in \NN$. We define the \defterm{$\ell$-Pingala numbers}, $\{P_{n}(\ell)\}_{n\geq 1}$, as the sequence generated by the $\ell+1$ term recurrence
    \[P_n(\ell)=P_{n-1}(\ell)+P_{n-2}(\ell)+\cdots +P_{n-1-\ell}(\ell),\qquad \mbox{for $n\geq \ell+2$}\]
and the initial values $P_j(\ell)=2^{j-1}$ for $j\in[\ell+1]$.
\end{definition}

\Cref{thm:recurrences for weakly increasing ell frs} and \Cref{def:ell-Pingala} imply that for all $n,\ell\in\NN$ with $\ell+1\leq n$,  
\begin{align}
|\FR{n}^\uparrow(\ell)|=P_n(\ell).\label{eq:pingala count}
\end{align}

\begin{remark}  \label{remark:Pingala} 
Pingala was an Indian mathematician and poet who used combinatorial methods to describe meter in Sanskrit poetry. Pingala described a form of the Fibonacci numbers and his Mount Meru Triangle is equivalent to Pascal's Triangle. His work continues to inform research into the connection between music and mathematics, such as in the work of Hall and Kre\v{s}imir \cite{MR1836944}. Kumar \cite{Kumar2021FibonacciSG} discusses a variety of generalizations of the Fibonacci numbers. Based on his research, Kumar advocates for calling the Fibonacci numbers the Pingala-Fibonacci numbers and Pascal's Triangle the Pingala-Pascal Triangle. The $k$-generalized Fibonacci numbers defined by the same recurrence in \Cref{def:ell-Pingala} were studied by Miles, Jr. (\cite{kfibonacci}), but our $\ell$-Pingala numbers are the specialization obtained by choosing increasing powers of two for the seed values. The $1$-Pingala numbers are precisely the Fibonacci numbers with initial values 1 and 2.
\end{remark} 

Next we show that if $\pi\in\Sym_n$ and $\tau$ is a run of $\pi$, then restricting the bijection $\varphi$ (in \Cref{lem:save us}) to the subset 
\[\FR{\tau}(\ell)=\left\{\alpha|_{\tau}:\alpha\in\out_{\FR{n}(\ell)}^{-1}
(\pi)\right\}\subseteq\FR{\tau}\]
is a bijection to the set of weakly increasing $\ell$-interval Fubini rankings with $\len(\tau)$ competitors, i.e., the set $\FR{\len(\tau)}^\uparrow(\ell)$.
\begin{lemma}\label{restricting to ell-interval is still cool}
Fix $\pi\in\Sym_n$ and let $\tau$ be a run of $\pi$. Then the sets $\FR{\tau}(\ell)$ and $\FR{\len(\tau)}^{\uparrow}(\ell)$ are in bijection.
\end{lemma}
\begin{proof}
    Let $k = \len(\tau)$. 
    Consider the map \[\phi=\varphi|_{\FR{\tau}(\ell)}:\FR{\tau}(\ell)\to \FR{k}^\uparrow(\ell)\]
    defined on $\beta=(b_1,b_2,\ldots,b_k)\in \FR{\tau}(\ell)$
    by
    \[\phi(\beta) = \varphi|_{\FR{\tau}(\ell)}((b_1, b_2, \ldots, b_k)) = (b_1 - (b_1-1), b_2 - (b_1-1), \ldots, b_k - (b_1-1)).\]
    We now show that $\phi$
    is a bijection.

    By \Cref{lem:save us}, to prove our claim, it suffices to show that the image of any 
    \[\beta=\alpha|_\tau=(b_1,b_2,\ldots,b_k)\in\FR{\tau}(\ell)\subseteq\FR{\tau}\] is a weakly increasing $\ell$-interval Fubini ranking with $k$ competitors. 
    Note that the weakly increasing part follows directly from the definition of $\varphi$, so it suffices to show that $\phi(\beta)\in\FR{k}(\ell)$.
By definition of the map $\phi$, equal entries in $\beta$ remain equal entries in the output $\phi(\beta)$. 
Since $\beta\in\FR{\tau}(\ell)$, 
    it is a tuple with at most $\ell+1$
entries that are equal. Therefore $\phi(\beta)$ also has at most $\ell+1$ entries that are equal, which completes the proof.
\end{proof}

By \Cref{corr: count of flat(pi) is 2^n-k}, we can enumerate Fubini rankings with a fixed parking outcome by counting the number of corresponding ordered set partitions whose 
parts have sizes given by the lengths of the runs of the parking outcome. 
However, with $\ell$-interval Fubini rankings, 
we must consider that the largest size of a block in a corresponding ordered set partition is no more than $\ell+1$. 
Whenever a run of a parking outcome is smaller than or equal to $\ell+1$, we can create a ordered set partition in that block without restricting the size of the blocks. 
When a run of a parking outcome is larger than $\ell+1$, the count is given by an $\ell$-Pingala number.
We prove this next.

\begin{theorem}\label{thm:big theorem}
        Fix $\ell\geq 1$. Let $\pi\in \mathfrak{S}_n$ have runs $\tau_1,\tau_2,\ldots,\tau_k$.
        Define $S_{\pi}$ to be the multiset containing the lengths of runs of $\pi$ smaller than or equal to $\ell+1$, and define $B_{\pi}$ to be the multiset containing the lengths of runs strictly larger than $\ell+1$: 
        \begin{align}
            S_\pi&=\{\len(\tau_i):\text{$\tau_i$ is a run in $\pi$ and }\len(\tau_i)\leq \ell+1\}\text{, and}\notag \\
            B_\pi&=\{\len(\tau_i):\text{$\tau_i$ is a run in $\pi$ and }\len(\tau_i)>\ell+1\}\notag.
        \end{align}
    Then 
        \begin{align}
        \left|\out_{\FR{n}(\ell)}^{-1}(\pi)\right|&=\left[\prod_{s\in S_\pi}2^{s-1}\right]\cdot
        \left[\prod_{b\in B_\pi}P_{b}(\ell)\right].
    \end{align}
\end{theorem}

\begin{proof}
Fix $\pi\in \Sym_n$ with runs $\tau_1,\tau_2,\ldots,\tau_k$ having corresponding lengths $\len(\tau_1), \len(\tau_2),\ldots,\len(\tau_k)$.
For $i\in[k]$, we denote the $j$-th entry in the run $\tau_i$ by 
$\tau_{(i,j)}$, where $j\in [\len(\tau_i)]$.

Given any element 
$\alpha\in\out_{\FR{n}(\ell)}^{-1}(\pi)$, there exist elements 
\begin{align*}
    \beta_1&=(b_{(1,1)},b_{(1,2)},\ldots,b_{(1,\len(\tau_1))})=\alpha|_{\tau_1}\in\FR{\tau_1}(\ell),\\
    \beta_2&=(b_{(2,1)},b_{(2,2)},\ldots,b_{(2,\len(\tau_2))})=\alpha|_{\tau_2}\in\FR{\tau_2}(\ell),\\
    &\hspace{2mm}\vdots\\
    \beta_k&=(b_{(k,1)},b_{(k,2)},\ldots,b_{(k,\len(\tau_k))})=\alpha|_{\tau_k}\in\FR{\tau_k}(\ell),
\end{align*} 
such that $\alpha$ has entry $b_{(i,j)}$ in index $\tau_{(i,j)}$, for each $i\in[k]$ and for each $j\in[\len(\tau_i)]$.

Since the entries of $\tau_1,\tau_2,\ldots,\tau_k$ partition the indexing set $[n]$, counting the number of $\alpha$'s is equivalent to the product of the number of each possible $\beta_i$, with $i\in[k]$.
Therefore
\begin{equation}
    \left|\out_{\FR{n}(\ell)}^{-1}(\pi)\right| = \prod_{i=1}^k |\FR{\tau_i}(\ell)|.\label{AAAA}
\end{equation}
By \Cref{restricting to ell-interval is still cool}, for each $i\in[k]$, 
\begin{equation}
    |\FR{\tau_i}(\ell)|=| \FR{\len(\tau_i)}^\uparrow(\ell)|.\label{almost there}
\end{equation}
Substituting \Cref{almost there} into \Cref{AAAA} yields
\begin{equation}
    \left|\out_{\FR{n}(\ell)}^{-1}(\pi)\right|=\prod_{i=1}^k | \FR{\len(\tau_i)}^\uparrow(\ell)|.\label{almost almost done}
\end{equation}
Then,
by \Cref{eq:pingala count}, we have
$|\FR{\len(\tau)}^\uparrow(\ell)| = P_{\len(\tau)}(\ell)$, and substituting into \Cref{almost almost done} yields 
\begin{align}
\left|\out_{\FR{n}(\ell)}^{-1}(\pi)\right|=\prod_{i=1}^k P_{\len(\tau)}(\ell).\label{truly we are almost there}
\end{align}

Now define $S_{\pi}$ and $B_\pi$ as in the statement of the theorem. 
By \Cref{def:ell-Pingala}, when $s \in S_\pi$, we have $P_{s}(\ell)=2^{s-1}$. Using this fact, we replace in \Cref{truly we are almost there} the Pingala number with the appropriate power of two for any $s\in S_{\pi}$. Hence 
\[|\out^{-1}_{\FR{n}(\ell)}(\pi)|=\prod_{i=1}^kP_{\len(\tau_i)}(\ell)=\left[\prod_{s\in S_\pi}2^{s-1}\right]\cdot
        \left[\prod_{b\in B_\pi}P_{b}(\ell)\right],\]
which completes the proof.\qedhere
\end{proof}

Next we illustrate the use of \Cref{thm:big theorem} via an example.

\begin{example}
    Consider $\pi = 1342$ and $\ell = 2$. In this case, $\tau_1 = 1  3  4$, $\tau_2 = 2$, so $B_\pi = \{3\}$, $S_\pi = \{1\}$. Thus,
    \[\prod_{i=1}^2P_{\len(\tau_i)}(2)=\left[\prod_{s \in S_\pi} 2^{s-1} \right]\cdot\left[\prod_{b \in B_\pi} P_{b}(\ell)\right]= 2^{1-1} \cdot   P_3(2) = 4\]
    and, in fact, there are only four $2$-interval Fubini rankings of length $4$ with  parking outcome $\pi=1342$. They are:
    \[\out_{\FR{4}(2)}^{-1}(\pi)=\{(1,4,1,1),(1,4,1,3), (1,4,2,2), (1,4,2,3)\}.\]
\end{example}
Next we give a formula for the number of $\ell$-interval Fubini rankings.

\begin{corollary}\label{cor:formula for ell fubini using compositions}
     If $n\in \NN$, then 
     \begin{align}
     |\FR{n}(\ell)|
     &= \sum_{k=1}^{n}\left(\sum_{\textbf{c}=(c_1,c_2,\ldots,c_k)\in\Comp(n,k)}\left(A(n,\textbf{c})\cdot\prod_{\substack{c_i\in \textbf{c} \\ c_i\leq \ell+1}}2^{c_i-1}\cdot\prod_{\substack{c_i\in \textbf{c}\\c_i>\ell+1}}P_{c_i}(\ell)\right)\right),
     \end{align}
     where $A(n,\textbf{c})$ is the number of permutations $\pi\in\Sym_n$ with runs of lengths given by $\textbf{c}$, and $P_i(\ell)$ denotes the $i$-th $\ell$-Pingala number, and if a product is empty, then we set the product equal to 1.
\end{corollary}
\begin{proof}
By definition 
\begin{equation}
    |\FR{n}(\ell)|=\sum_{\pi\in\Sym_n}|\out_{\FR{n}(\ell)}^{-1}(\pi)|.\notag
\end{equation}
By \Cref{thm:big theorem}, we know that $|\out_{\FR{n}(\ell)}^{-1}(\pi)|$ is a product of powers of two and $\ell$-Pingala numbers which are indexed by the lengths of the runs making up $\pi$.
Hence, given $\textbf{c}=(c_1,c_2,\ldots,c_k)\in\Comp(n,k)$,  there are $A(n,\textbf{c})$ permutations $\pi$ with runs of lengths given by $\textbf{c}$, and each of those permutations has the same number of $\ell$-interval Fubini rankings parking in that order. 
Thus, we must account for all possible compositions of run lengths and we must also vary the number of runs from 1 to $n$, which yields
\begin{align*}
|\FR{n}(\ell)|&=\sum_{\pi\in\Sym_n}|\out_{\FR{n}(\ell)}^{-1}(\pi)|\\
&=\sum_{k=1}^{n}\left(\sum_{\textbf{c}=(c_1,c_2,\ldots,c_k)\in\Comp(n,k)}\left(A(n,\textbf{c})\cdot\prod_{\substack{c_i\in \textbf{c}\\c_i\leq \ell+1}}2^{c_1-1}\cdot\prod_{\substack{c_i\in \textbf{c}\\c_i> \ell+1}}P_{c_i}(\ell)\right)\right),
\end{align*}
as claimed.
\end{proof}

We conclude with an example.

\begin{example}
We want to count the number of $2$-interval Fubini rankings with $5$ competitors. 
Constructing these rankings is computationally demanding, and we have written Sage code \cite{github} that constructs $\ell$-interval Fubini rankings of length $n$.
Using this code one can verify that $|\FR{5}(2)|=530$.

Next, we verify this count using \Cref{cor:formula for ell fubini using compositions}. 
We trust the reader can verify the following computations:
\begin{itemize}[leftmargin=.2in]
     \item $\Comp(5,1)=\{(5)\}$ and $A(5,(5))=1$.
     
     \item  $\Comp(5,2)=\{(4,1), (3,2), (2,3), (1,4)\}$ and \\ $A(5,(4,1))=4$, $A(5,(3,2))= 9$, $A(5,(2,3))=9$, $A(5,(1,4))=4$.

     \item  $\Comp(5,3)=\{(1,1,3), (1,3,1), (3,1,1),(1,2,2),(2,1,2),(2,2,1)\}$ and  $A(5,(1,1,3))=6$, $A(5,(1,3,1))=11$, $A(5,(3,1,1))=6$, $A(5,(1,2,2))=16$, $A(5,(2,1,2))=11$, $A(5,(2,2,1))=16$.

   \item $\Comp(5,4)=\{(1,1,1,2), (1,1,2,1), (1,2,1,1),(2,1,1,1)\}$, and $A(5,(1,1,1,2))=4$, $A(5,(1,1,2,1))=9$, $A(5,(1,2,1,1))=9$, $A(5,(2,1,1,1))=4$.

   \item $\Comp(5,5)=\{(1,1,1,1,1)\}$ and $A(5,(1,1,1,1,1))=1$.
\end{itemize}
Therefore
\begin{align*}
|\FR{5}(2)|&=\sum_{k=1}^{5}\left(\sum_{\textbf{c}=(c_1,c_2,\ldots,c_k)\in\Comp(5,k)}\left(A(5,\textbf{c})\prod_{\substack{c_i\in \textbf{c}\\c_i\leq 3}}2^{c_i-1}\prod_{\substack{c_i\in \textbf{c}\\c_i> 3}}P_{c_i}(2)\right)\right)\\
&= A(5,(5))  (P_5(2)) +
 A(5,(4,1))  (2^{1-1})  (P_4(2)) 
+ A(5,(3,2))  (2^{3-1}  2^{2-1}) \\
&\qquad+ A(5,(2,3))  (2^{2-1}  2^{3-1}) + A(5,(1,4))  (2^{1-1})  (P_4(2))
+ A(5,(1,1,3))  (2^{1-1}  2^{1-1}  2^{3-1}) \\
&\qquad+ A(5,(1,3,1))  (2^{1-1}  2^{3-1}  2^{1-1}) 
+A(5,(3,1,1))  (2^{3-1}  2^{1-1}  2^{1-1}) \\
&\qquad+ A(5,(1,2,2))  (2^{1-1}  2^{2-1}  2^{2-1})
+ A(5,(2,1,2))  (2^{2-1}  2^{1-1}  2^{2-1}) \\
&\qquad+ A(5,(2,2,1))  (2^{2-1}  2^{2-1}  2^{1-1}) + A(5,(1,1,1,2))  (2^{1-1}  2^{1-1}  2^{1-1}  2^{2-1}) \\
&\qquad + A(5,(1,1,2,1))  (2^{1-1}  2^{1-1}  2^{2-1}  2^{1-1})
+ A(5,(1,2,1,1))  (2^{1-1}  2^{2-1}  2^{1-1}  2^{1-1}) \\
&\qquad+ A(5,(2,1,1,1))  (2^{2-1}  2^{1-1}  2^{1-1}  2^{1-1}) + A(5,(1,1,1,1,1))  (2^{1-1}  2^{1-1}  2^{1-1}  2^{1-1}  2^{1-1}) \\
&= (1 \cdot 1 \cdot 13) 
+ (4 \cdot 1 \cdot 7) 
+ (9 \cdot 4 \cdot 2) 
+ (9 \cdot 4 \cdot 2) 
+ (4 \cdot 1 \cdot 7) 
+ (6 \cdot 4 \cdot 1) \\
&\qquad+ (11 \cdot 4 \cdot 1) 
+ (6 \cdot 4 \cdot 1) 
+ (16 \cdot 4 \cdot 1) 
+ (11 \cdot 4 \cdot 1) 
+ (16 \cdot 4 \cdot 1) 
+ (4 \cdot 2 \cdot 1) \\
&\qquad+ (9 \cdot 2 \cdot 1) 
+ (9 \cdot 2 \cdot 1) 
+ (4 \cdot 2 \cdot 1) 
+ (1 \cdot 1 \cdot 1) \\
&= 13+28+72+72+28+24+44+24+64+44+64+8+18+18+8+1 \\
&= 530,
\end{align*}
as expected. 
\end{example}

\section{Future Work}\label{sec:future}
We conclude with some directions for further study.

\subsection{$(m,n)$-Fubini Rankings}
If $m,n\in \mathbb{N}$ with $n\leq m$, then  $(m,n)$-parking functions are parking functions  where there are $n$ cars and $m$ parking spots. It would be interesting to study the analogous concept of $(m,n)$-Fubini rankings, where there are $m$ possible ranks and only $n$ competitors, and we can use at most $n$ of the $m$ possible ranks. These rankings would also ensure that whenever $k$ competitors tie at rank $i$, the ranks $i+1,i+2,\ldots,i+k-1$ are disallowed.
The parking outcome of a $(m,n)$-Fubini ranking would be a permutation of $[n]$ with $m-n$ bars to indicate the empty parking spots. 
One could study the analogous question.
\begin{question}
    How many $(m,n)$-Fubini rankings are there with a fixed parking outcome?
\end{question}
\subsection{$b$-Fubini Rankings}
For $b\in[n]$, the set of $b$-Fubini rankings with $n$ competitors is the collection of elements in $\FR{n}$ in which the first $b$ competitors place in distinct ranks. 
The parking outcome is  a permutation of $[n]$ and so we ask again: 
\begin{question}
    How many $b$-Fubini rankings are there with a fixed parking outcome?
\end{question}

\subsection{Fubini Rankings with a Fixed Maximum Rank}
Fix $m\in[n]$ and let 
\[\MFR_n(m)=\left\{\alpha=(a_1,a_2,\ldots,a_n)\in \FR{n}:\max_{1\leq i\leq n}(a_i)=m\right\}\] denote the set of Fubini rankings of length $n$ with maximum rank being $m$. The number of these rankings is given next.

\begin{lemma}
    If $1\leq m \leq n$, then 
    \begin{equation}
        |\MFR_n(m)| = \binom{n}{n-m+1}\Fub_{m-1}.
    \end{equation}
\end{lemma}
\begin{proof}
    Fix $m \in [n]$. Let $\alpha\in\MFR_n(m)$ be arbitrary. 
    
    Let $\alpha^{\uparrow} = (a'_1, a'_2, \ldots, a'_n)$ be the weakly increasing rearrangement of $\alpha$. By \Cref{lemma: weakly increasing Fubini proof}, we have that $a'_m = m$. Since the maximum rank in $\alpha$ is $m$, we then have that the next $n-m$ entries are also $m$, with any previous entries less than $m$. Therefore, there are $n-m+1$ entries in $\alpha$ labeled $m$.
    Let $\beta=(b_1,b_2,\ldots,b_{n-m+1})$ denote $\alpha$ with the entries labeled $m$ removed. Suppose for the sake of contradiction that $\beta$ is not a Fubini ranking. This implies that by \Cref{lemma: weakly increasing Fubini proof}, there is some $i\in[n-m+1]$ such that
    $b_i \notin \{b_{i-1}, i\}$. However $\alpha$ is a Fubini ranking, which implies that the first $n-m+1$ entries in $\alpha^\uparrow$ satisfy $a'_i=a'_{i-1}$ or $a'_i=i$, which is a contradiction. Therefore, if $\alpha\in\MFR_n(m)$, then $\alpha$ is a tuple consisting of $n-m+1$ entries labeled $m$, and the remaining $m-1$ entries form a Fubini ranking of length $m-1$. Therefore the set $\MFR_n(m)$ is enumerated by,
    \begin{equation}
        |\MFR_n(m)|=\binom{n}{n-m+1}\Fub_{m-1}.\qedhere
    \end{equation}
    which completes the proof.
\end{proof}
With the count of $\MFR_n(m)$ now known, we pose the following question:
\begin{question}
    How many Fubini rankings are there with a fixed maximum rank and a fixed parking outcome?
\end{question}

\subsection{$k$-Restricted Fubini Rankings} Fix $k \in [n]$ and let $\RFR_n(k)$ be the set of Fubini rankings of length $n$ which contain \textit{exactly} $k$ ties (of any length larger than 1). That is 
\[\RFR_n(k) = \{\alpha \in \FR{n} : |\{B \in \psi(\alpha) : |B| > 1\}| = k\},\]
where $\psi$ is the bijection between Fubini rankings and ordered set partitions described in \Cref{lem:bijection between Fubini rankings and ordered set partitions}. 

\begin{question}\label{question:restricted fubs}
    How many restricted Fubini rankings of length $n$ with exactly $k$ ties are there? Moreover, how many are there with a fixed parking outcome $\pi$?
\end{question}
The fact that $\bigcup_{k=1}^{n} \RFR_n(k) = \FR{n}$, shows that an answer to \Cref{question:restricted fubs} gives a new formula for the Fubini numbers. 

\bibliographystyle{plain}
\bibliography{bibliography.bib}

\end{document}